\documentclass[11pt, reqno]{amsart}
\usepackage{latexsym}		      
\usepackage{amssymb}                  
\usepackage{bm}             
\usepackage{version}
\usepackage{psfrag,graphicx}
\newtheorem{Theorem}{Theorem}[section] 
\newtheorem{Lemma}[Theorem]{Lemma} 
\newtheorem{Corollary}[Theorem]{Corollary}

\newtheorem{Definition}[Theorem]{Definition}
\newtheorem{Remark}[Theorem]{Remark}
\newtheorem{Conjecture}[Theorem]{Conjecture}

\newcommand{\T}{\mathbb{T}}
\newcommand{\D}{\mathbb{D}}
\newcommand{\E}{\mathbb{E}}
\newcommand{\la}{\lambda}
\newcommand{\al}{\alpha}
\newcommand{\ph}{\varphi}
\newcommand{\ep}{\varepsilon}
\newcommand{\up}{\upsilon}
\newcommand{\C}{\mathbb{C}}
\newcommand{\R}{\mathbb{R}}
\newcommand{\nn}{\nonumber}
\newcommand{\ds}{\displaystyle}
\newcommand{\dia}{\diamond}
\newcommand{\dE}{b\E}
\renewcommand{\S}{\mathcal{S}}
\providecommand{\Aut}{\mathop{\rm Aut}\nolimits}
\providecommand{\diag}{\mathop{\rm diag}\nolimits}
\newcommand\nin{\noindent}
\newcommand{\ip}[2]{\left\langle #1, #2 \right\rangle}
\numberwithin{equation}{section}

\begin{document}

\author{A. A. Abouhajar, M. C. White and N. J. Young }

\title[A Schwarz lemma]
{A Schwarz lemma for a domain related to mu-synthesis}

\subjclass{30C80, 32F45, 93D21}

\keywords{Interpolation, invariant distance, automorphism, distinguished boundary, structured Nevanlinna-Pick problem, robust stabilization}

\begin{abstract}
We prove a Schwarz lemma for a domain  $\E$ in $\C^3$ that arises in connection with a problem in $H^\infty$ control theory.  We describe a class of automorphisms of $\E$ and determine the distinguished boundary of $\E$.  We apply our Schwarz lemma to a special case of the $\mu$-synthesis problem.
\end{abstract}
\maketitle
\section{Introduction} \label{intro}
In this paper we study the complex geometry of a domain $\E \subset \C^3$ which is relevant to some problems of analytic interpolation that arise in control engineering.
  Our main result is a Schwarz lemma for $\E$, but we also identify a natural class of automorphisms of $\E$ and determine the distinguished boundary of $\E$.
\begin{Definition} \label{defE}
The {\em tetrablock} is the domain
\[
\E = \{ x \in \C^3: \quad 1-x_1 z - x_2 w + x_3 zw \neq 0 \mbox{  whenever }
|z| \le 1, |w| \le 1 \}.
\]
The closure of $\E$ is denoted by $\bar\E$.
\end{Definition}
$\E$ is a polynomially convex, non-convex domain, is starlike about $0$ and intersects $\R^3$ in a regular tetrahedron (Theorems \ref{starlike}, \ref {EcapR3} and \ref{polyconvex}).  To a first approximation one can think of $\E$ as the set of linear fractional maps $(x_3 z - x_1)/(x_2 z-1)$ that map the closed unit disc $\Delta$ into the open unit disc $\D$, but this viewpoint, though useful, must be interpreted with care since it does not capture the case that $x_1x_2=x_3$: see Theorem \ref{characE} for a precise statement.

Here is our main result.   To cut down on subscripts we write the typical point of $\E$ as $(a,b,p)$.
\begin{Theorem} \label{schwarzL}
Let $\la_0 \in \D\setminus \{0\}$ and let $x=(a,b,p) \in \E$.  The following conditions are equivalent:
\begin{enumerate}
\item[\rm(1)] there exists an analytic function $\ph: \D \to \bar\E$ such that
$\ph(0)=(0,0,0)$ and $\ph(\la_0) = x$;
\item[\rm(1$^\prime$)] there exists an analytic function $\ph: \D \to \E$ such that
$\ph(0)=(0,0,0)$ and $\ph(\la_0) = x$;
\item[\rm(2)]
\[
\max \left\{ \frac{|a-\bar b p|+ |ab-p|}{1-|b|^2}, 
\frac{|b-\bar a p|+|ab-p|}{1-|a|^2} \right\} \leq |\la_0|;
\]
\item[\rm(3)] either $|b| \leq |a|$ and 
\[
\frac{|a-\bar b p|+ |ab-p|}{1-|b|^2} \leq |\la_0|
\]
or $|a| \leq |b|$ and
\[
\frac{|b-\bar a p|+|ab-p|}{1-|a|^2} \leq |\la_0|;
\]
\item[\rm(4)] there exists a $2\times 2$ function $F$ in the Schur class such that
\[
F(0) = \left[ \begin{array}{cc} 0 & * \\ 0 & 0 \end{array} \right] \mbox{ and }
F(\la_0) = A = [a_{ij}]
\]
where $x = (a_{11},a_{22}, \det A).$
\end{enumerate}
\end{Theorem}
Recall that the {\em Schur class} (of type $m \times n$) is the set of analytic functions $F$ on $\D$ with values in the space $\C^{m\times n}$ of complex $m\times n $ matrices such that $||F(\la)|| \le 1$ for all $\la \in \D$; here and elsewhere $||.||$ denotes the usual operator norm (the largest singular value) of a matrix.

The starting point of our research was a certain special case of the {\em $\mu$-synthesis problem}, which arises in the $H^\infty$ approach to the problem of robust control 
\cite{doyle, DP}.   Perhaps the most appealing still unsolved instance of $\mu$-synthesis is the spectral Nevanlinna-Pick problem (to construct an analytic square-matrix valued function on the disc subject to interpolation conditions and a bound on the spectral radius, \cite{BFT1, AY1}). 
In earlier work \cite{AY1,AY3} some progress was made on the  $2\times 2$ spectral Nevanlinna-Pick problem via the analysis of a domain in $\C^2$ known as the symmetrised bidisc.  It transpired that this domain and its higher dimensional analogues have a rich geometry and function theory, of interest independently of their connections with engineering (for example \cite{AY4, cos,EZ,PZ,bharali} among others).  In an analogous way the study of the ``next'' special case of $\mu$-synthesis for $2\times 2$ matrix functions led us to  analyse the tetrablock.    In Section \ref{mu} we explain the connection between $\E$ and $\mu$-synthesis and give an application of our Schwarz lemma.  Note the interesting fact that $\mu$-synthesis problems can be ill conditioned (Remark \ref{muschw}(iv)).

In Section \ref{charact} we give a variety of characterizations of the open and closed tetrablocks and present some basic geometric properties of $\E$.  In Section \ref{schwarz} we prove Theorem \ref{schwarzL} and deduce a formula for the Carath\'eodory and Kobayashi distances of a general point of $\E$ from the origin.  In Section \ref{uniq}   we  show that there is no uniqueness statement for the extremal case and
in Section \ref{all} we describe all solutions of a Schwarz-type 2-point interpolation problem for $\E$.
In Section \ref{autos} we identify a rich class of automorphisms of $\E$.  In Section \ref{shilov} we calculate the distinguished boundary of $\E$.   In Section \ref{retract} we pose the question as to whether $\E$ is an analytic retract of a certain convex domain and prove a partial negative result.

 We write $\T$ for the unit circle in $\C$.  As usual, $H^\infty$ denotes the Banach space of bounded analytic functions on $\D$ with supremum norm.  An {\em automorphism} of a domain $\Omega$ is a biholomorphic self-map of $\Omega$; the automorphism group of  $\Omega$ will be denoted by $\Aut \Omega$.    We denote by $\S_{m\times n}$ the class (slightly smaller than the Schur class) of analytic functions $F:\D \to \C^{m\times n}$ such that $||F(\la)|| < 1$ for all $\la \in \D$.

For $Z \in \C^{m\times n}$ such that $||Z|| < 1$ we denote by $\mathcal{M}_Z$ the matricial M\"obius transformation defined for contractive $X\in \C^{m\times n}$ by
\[
\mathcal{M}_Z(X) = (1-ZZ^*)^{-\tfrac 12}(X-Z)(1-Z^*X)^{-1}(1-Z^*Z)^{\tfrac 12}.
\]
Recall that $\mathcal{M}_Z^{-1} = \mathcal{M}_{-Z}$ as self-mappings of the closed unit ball of $\C^{m\times n}$.
We shall denote the $(i,j)$ entry of a matrix $A$ by $[A]_{ij}$.

This paper is based on the first-named author's Ph.D. thesis \cite{alaa}.

\section{Characterization of the tetrablock}  \label{charact}
The following rational functions of $4$ variables play a central role in the study of 
$\E$.
\begin{Definition} \label {defPsi}
For $z \in \C$ and $x=(x_1,x_2,x_3) \in \C^3$ we define
\begin{eqnarray}
\Psi(z,x) &=& \frac{x_3 z - x_1}{x_2z - 1}, \label{dPsi}\\
\Upsilon(z,x) &=& \Psi(z,x_2,x_1,x_3) = \frac{x_3z-x_2}{x_1z-1}, \label{dUps}\\
 D(x) &=& \sup_{z \in \D} | \Psi(z,x)| = || \Psi(.,x)||_{H^\infty}. \label{dD}
\end{eqnarray}
\end{Definition}
We interpret $\Psi(.,x)$ to be the constant function equal to $x_1$ in the event that
$x_1 x_2 = x_3$; thus $\Psi(z,x)$ is defined when $zx_2 \ne 1$ or $x_1 x_2 = x_3$.
The quantity $D(x)$ is finite (and $\Psi(.,x) \in H^\infty$) if and only if either $x_2 \in \D$ or $x_1x_2=x_3$.
Indeed, for $x_2 \in \D$, the linear fractional function $\Psi(.,x)$ maps $\D$ to
the open disc with centre and radius
\begin{equation}  \label{imPsi}
\frac{x_1-\bar x_2 x_3}{1-|x_2|^2}, \quad  \frac{|x_1 x_2 - x_3|}{1-|x_2|^2}
\end{equation}
respectively.  Hence
\begin{equation}  \label{formD}
D(x) = \left \{ \begin{array}{ll} \displaystyle \frac {|x_1 - \bar x_2 x_3| + |x_1 x_2 - x_3|}
{1-|x_2|^2} & \mbox{ if $|x_2| < 1$} \\
|x_1| & \mbox{ if $x_1x_2=x_3$}\\
\infty & \mbox { otherwise.} \end{array}
\right .
\end{equation} 
Similarly, if $x_1 \in \D$, $\Upsilon(.,x)$ maps $\D$ to the open disc with 
centre and radius 
\[
\frac{x_2 - \bar x_1 x_3}{1-|x_1|^2}, \quad  \frac{|x_1 x_2 - x_3|}{1-|x_1|^2}
\]
respectively.
\begin{Theorem} \label{characE}
For $x \in \C^3$ the following are equivalent.
\begin{enumerate}
\item[\rm(1)] $x \in \E;$
\item[\rm(2)] $||\Psi(.,x)||_{H^\infty} < 1$ and if $x_1 x_2=x_3$ then, in addition,
$|x_2| <1$;
\item[\rm(2$^\prime$)] $||\Upsilon(.,x)||_{H^\infty} < 1 $ and if $x_1 x_2=x_3$ then, in addition,
$|x_1| <1$;
\item[\rm (3)] $|x_1 - \bar x_2 x_3| + |x_1 x_2 - x_3| < 1-|x_2|^2$;
\item[\rm (3$^\prime$)] $|x_2 - \bar x_1 x_3| + |x_1 x_2 - x_3| < 1-|x_1|^2$;
\item[\rm (4)] $|x_1|^2 - |x_2|^2 + |x_3|^2 +2|x_2 - \bar x_1 x_3| <1$ and $|x_2| < 1$;
\item[\rm (4$^\prime$)] $- |x_1|^2 + |x_2|^2 + |x_3|^2 +2|x_1 - \bar x_2x_3| <1$ and $|x_1| < 1$;
\item[\rm(5)] $|x_1|^2 + |x_2|^2 - |x_3|^2 + 2|x_1x_2 - x_3| < 1$ 
and $|x_3| < 1$;
\item[\rm(6)] $|x_1 - \bar x_2x_3| + |x_2 - \bar x_1 x_3| <1 - |x_3|^2$;
\item[\rm(7)] there exists a $2 \times 2$ matrix $A=[a_{ij}]$ such that $||A|| <1$  and $x=(a_{11}, a_{22}, \det A)$;
\item[\rm(8)] there exists a symmetric $2 \times 2$ matrix $A=[a_{ij}]$ such that $||A|| <1$  and $x=(a_{11}, a_{22}, \det A)$;
\item[\rm(9)] $|x_3| < 1$ and there exist $\beta_1, \beta_2 \in \C$ such that
$|\beta_1| + |\beta_2| < 1$  and
\[
x_1=\beta_1 + \bar\beta_2 x_3, \quad x_2 = \beta_2+\bar\beta_1 x_3.
\]
\end{enumerate}
\end{Theorem}
\begin{proof}
Consider the case that $x_1x_2=x_3$: conditions (1) to (8) (apart from (6)) easily reduce to the pair of statements $|x_1| < 1,  |x_2| < 1$. 
Hence we may suppose that $x_1x_2 \neq x_3$
for proof of the equivalence of statements (1) to (5), (7) and (8).  It is clear that $\E$ is symmetric in its first two variables: $(x_1,x_2,x_3) \in \E$ if and only if $(x_2,x_1,x_3) \in \E$.  Hence, if we show that (1) $\Leftrightarrow$ (2)  then it will follow also that (1) $\Leftrightarrow$ (2$^\prime$) since
$\Upsilon(.,x)=\Psi(.,x_2,x_1,x_3)$.  We shall prove 
\[
\begin{array}{ccccc}
(1)&\Leftrightarrow & (2) & \Leftrightarrow & (3)\\
     &         &  \Updownarrow &  & \\
      &	&   (4) &  &
\end{array}
\mbox{   and then   }
\begin{array}{ccccc}
(1) & \Leftrightarrow  & (5) & \Leftarrow & (7)\\
      &                    & \Downarrow &  \textrm{\rotatebox{38}{$\Rightarrow$}}& \\
	&			& (8)  & &
\end{array}
\mbox{ and } 
\begin{array}{ccc}
(1) & \Leftarrow  & (9) \\
 \Downarrow &  \textrm{\rotatebox{38}{$\Rightarrow$}} & \\
 (6) &  &
\end{array}
\]
and the equivalences ($n^\prime$) follow by symmetry.\\
(1)$\Leftrightarrow$(2)  Condition (1) is equivalent to 
\[
z(x_1 - x_3 w) \neq 1 - x_2 w \mbox { for all } z, w \in \Delta,
\]
that is, $|x_2| < 1$ and $ 1 \notin z\Psi(\Delta,x)$ for all $z \in \Delta$.
Hence (1) holds if and only if $\Psi(\Delta,x)$ does not meet the complement of $\D$,
which is so if and only if (2) holds.\\
(2)$\Leftrightarrow$(3)  By equation (\ref{formD}),
\[
||\Psi(.,x)||_{H^\infty} = D(x) = 
\frac{|x_1 - \bar x_2 x_3| +|x_1x_2-x_3|}{1-|x_2|^2},
\]
from which the equivalence is immediate.\\
(2)$\Leftrightarrow$(4)   By the maximum principle, (2) holds if and only if $|x_2| < 1$ and
\[
|x_3 z - x_1|^2 < |x_2 z - 1|^2  \mbox{ for all } z \in \T.
\]
On expanding and re-arranging we find that (2)$\Leftrightarrow$(4).\\
(1)$\Leftrightarrow$(5)  The left hand side of (4$^\prime$) is unchanged if $x_2,x_3$ are replaced by $\bar x_3, \bar x_2$ respectively. Hence $(x_1,x_2,x_3) \in \E$ if and only if
$(x_1,\bar x_3, \bar x_2) \in \E$, which, by the equivalence (1)$\Leftrightarrow$(4), is so if and only if (5) holds.

The following is a routine calculation.
\begin{Lemma} \label{AstarA}
If
\[
A=\left[ \begin{array}{cc} x_1 & b \\ c & x_2  \end{array} \right]
\]
where $bc = x_1x_2 - x_3$ then $\det A  = x_3$,
\begin{equation} \label{xcon}
1-A^*A = \left[ \begin{array}{cc} 1-|x_1|^2 - |c|^2 & -b\bar x_1 - \bar c x_2 \\
-\bar b x_1 - c \bar x_2 & 1 - |x_2|^2 - |b|^2 \end{array} \right]
\end{equation}
and
\begin{equation} \label{formdet}
\det(1 - A^*A) = 1 - |x_1|^2 - |x_2|^2 +|x_3|^2 -|b|^2 - |c|^2.
\end{equation}
\end{Lemma}
(5)$\Rightarrow$(8)$\Rightarrow$(7)$\Rightarrow$(5)
Suppose (5) holds.    Choose (either) $w$ such that $w^2=x_1x_2 - x_3$ and let
$A=\left[ \begin{array}{cc}x_1 & w\\ w & x_2 \end{array} \right]$.
Since
(5)$\Leftrightarrow$(1)$\Leftrightarrow$(4)$\Leftrightarrow$(4$^\prime$),
the diagonal entries of $1-A^*A$ are positive (see equation (\ref{xcon})), and by equation (\ref{formdet})
\[
\det(1-A^*A) = 1-|x_1|^2 - |x_2|^2 +|x_3|^2 -2|x_1x_2-x_3| >0.
\]
Hence $\|A\| < 1$ and so (5)$\Rightarrow$(8).

Trivially (8)$\Rightarrow$(7).  Suppose (7) holds.  Since
\[
|a_{12}|^2 + |a_{21}|^2 \geq 2|a_{12}a_{21}| = 2|x_1x_2-x_3|,
\]
we have 
\[
1 - |x_1|^2 - |x_2|^2 +|x_3|^2 - 2|x_1x_2 - x_3| \geq 
1 - |x_1|^2 - |x_2|^2 +|x_3|^2 -|a_{12}|^2 - |a_{21}|^2 = \det(1-A^*A) > 0.
\]
Thus (7)$\Rightarrow$(5).

For the remaining implications we do not assume $x_1x_2 \neq x_3$.\\
(1)$\Rightarrow$(6)$\Rightarrow$(9)$\Rightarrow$(1)
Suppose (1).  Then (4) and (4$^\prime$) hold, and on adding these two inequalities we obtain (6).

Now suppose (6).  Certainly $|x_3| < 1$.  Let
\begin{equation} \label{defbetas}
\beta_1 = \frac{x_1 - \bar x_2 x_3}{1-|x_3|^2}, \quad
\beta_2 = \frac{x_2 - \bar x_1 x_3}{1 - |x_3|^2}.
\end{equation}
Inequality (6) tells us that $|\beta_1|+|\beta_2| < 1$ and
it is immediate that
\[
\beta_1 + \bar\beta_2 x_3 = x_1, \quad \beta_2 + \bar\beta_1 x_3 = x_2.
\]
Hence (9) holds.

Suppose (9).  Then $|x_2| \leq |\beta_1| + |\beta_2| < 1$ and
\[
|x_1|^2 - |x_2|^2 = (|\beta_1|^2 - |\beta_2|^2)(1-|x_3|^2) \leq 
(|\beta_1| - |\beta_2|)(1 - |x_3|^2).
\]
Moreover $x_2 - \bar x_1 x_3 = \beta_2(1-|x_3|^2)$, and so
\[
|x_1|^2 - |x_2|^2 + 2|x_2 - \bar x_1 x_3| \leq (|\beta_1| -|\beta_2 | +2|\beta_2|)
(1-|x_3|^2) < 1 - |x_3|^2.
\]
Thus (9)$\Rightarrow$(4)$\Rightarrow$(1).
\end{proof}
There are analogous characterizations of $\bar\E$.
\begin{Theorem} \label{closE}
For $x \in \C^3 $ the following conditions are equivalent.
\begin{enumerate}
\item[\rm (0)] $1 - x_1z -x_2 w + x_3 zw \neq 0$ for all $z, w \in \D$;
\item[\rm(1)] $x \in \bar\E;$
\item[\rm(2)] $||\Psi(.,x)||_{H^\infty} \leq 1$ and if $x_1 x_2=x_3$ then, in addition,
$|x_2| \leq1$;
\item[\rm(2$^\prime$)] $||\Upsilon(.,x)||_{H^\infty} \leq 1 $ and if $x_1 x_2=x_3$ then, in addition, 
$|x_1| \leq1$;
\item[\rm (3)] $|x_1 - \bar x_2 x_3| + |x_1 x_2 - x_3| \leq 1-|x_2|^2$ and if $x_1x_2=x_3$ then, in addition, $|x_1| \leq 1$;
\item[\rm (3$^\prime$)] $|x_2 - \bar x_1 x_3| + |x_1 x_2 - x_3| \leq 1-|x_1|^2$ and if $x_1x_2=x_3$ then, in addition, $|x_2| \leq 1$;
\item[\rm (4)] $|x_1|^2 - |x_2|^2 + |x_3|^2 +2|x_2 - \bar x_1 x_3| \leq1$ and $|x_2| \leq 1$;
\item[\rm (4$^\prime$)] $- |x_1|^2 + |x_2|^2 + |x_3|^2 +2|x_1 - \bar x_2x_3| \leq 1$ and $|x_1| \leq 1$;
\item[\rm(5)] $|x_1|^2 + |x_2|^2 - |x_3|^2 + 2|x_1x_2 - x_3| \leq 1$ 
and $|x_3| \leq 1$;
\item[\rm(6)] $|x_1 - \bar x_2x_3| + |x_2 - \bar x_1 x_3| \leq 1 - |x_3|^2$ 
and if $|x_3|=1$ then, in addition, $|x_1| \leq 1$;
\item[\rm(7)] there exists a $2 \times 2$ matrix $A=[a_{ij}]$ such that $||A|| \leq 1$ and
$x=(a_{11}, a_{22}, \det A)$;
\item[\rm(8)] there exists a symmetric $2 \times 2$ matrix $A=[a_{ij}]$ such that $||A|| \leq 1$ and
$x=(a_{11}, a_{22}, \det A)$;
\item[\rm(9)] $|x_3| \leq 1$ and there exist $\beta_1, \beta_2 \in \C$ such that
$|\beta_1| + |\beta_2| \leq 1$  and
\[
x_1=\beta_1 + \bar\beta_2 x_3, \quad x_2 = \beta_2+\bar\beta_1 x_3.
\]
\end{enumerate}
\end{Theorem}
\begin{proof}
(0)$\Rightarrow$(1)  Suppose (0) and consider any $\zeta, \eta \in \Delta$.  Since 
$(r\zeta,r\eta) \in \D^2$ for any $r \in (0,1)$ we have 
$1-x_1r\zeta -x_2r\eta +x_3r^2\zeta\eta \neq 0$.  Hence $(rx_1,rx_2,r^2x_3) \in \E$
for $r \in (0,1)$, and so $x \in \bar\E$.

(1)$\Rightarrow$(0)  Suppose $x\in \bar\E$ but $1-x_1z-x_2w+x_3zw =0$ for some
$z,w \in \D$.  Then $z\Psi(w,x) = 1$ and so $|\Psi(w,x)|>1$.  However, 
$|\Psi(w,\xi)| < 1$ for all $\xi \in \E$, and since $x$ is a limit point of such $\xi$
we have $|\Psi(w,x)| \leq 1$, a contradiction.

Consider the case that $x_1x_2=x_3$.  Condition (0) reduces to $1-x_1z \neq 0,\, 1-x_2w\neq 0$ for all $z,w \in \D$, that is, to 
\begin{equation}\label{pair}
|x_1| \leq 1, \quad |x_2| \leq 1.  
\end{equation}
An analysis of cases shows that conditions (2) to (5) and (7) and (8) also reduce to this pair of inequalities. 
In particular, condition (5) becomes
\[
(1-|x_1|)(1-|x_2|) \geq 0 \mbox{ and } (1-|x_1|)+(1-|x_2|) \geq 0,
\]
which is equivalent to the relations (\ref{pair}).  Thus (0) to (5), (7) and (8) are
all equivalent in the case that $x_1x_2=x_3$.

In the case that $x_1x_2 \neq x_3$ the proof of equivalence of (0)-(5), (7) and (8) is much as for Theorem \ref{characE}.  It remains to prove (1)$\Rightarrow$(6)$\Rightarrow$(9)
$\Rightarrow$(1), whether or not $x_1x_2=x_3$.

(1)$\Rightarrow$(6)  We have $|x_1| \leq 1$, for example from (2$^\prime$), and on adding the inequalities in (4) and (4$^\prime$) we deduce (6).\\
(6)$\Rightarrow$(9)  Suppose (6).  Clearly $|x_3| \leq 1$.
If $|x_3| < 1$ then the proof that (6)$\Rightarrow$(9) in Theorem \ref{characE} still applies.  If $|x_3|=1$ then $x_1 = \bar x_2 x_3, \, |x_2|=|x_1| \leq 1$ and we find that (9) holds with $\beta_1 = t x_1, \beta_2 = (1-t) x_2$ for any $t \in [0,1]$.\\
(9)$\Rightarrow$(1) is proved just as in Theorem \ref{characE}.
\end{proof}
\begin{Remark}  \rm (i) Further criteria for membership of $\E$ and $\bar\E$, in terms of
the structured singular value, are given in Theorem \ref{mult1} below.\\
(ii) Note the strange symmetry of $\E$ and $\bar\E$:
\[
(x_1,x_2,x_3) \mapsto (x_1,\bar x_3, \bar x_2)
\]
which we used in the proof and which follows from criterion (4$^\prime$).\\
(iii)  In relation to conditions (8) we observe that, for any $x\in\C^3$, there is either a unique symmetric $2\times 2$ matrix $A$ such that $x=(a_{11},a_{22},\det A)$ (when $x_1x_2=x_3$), or precisely two such $A$s, corresponding to the square roots of $x_1x_2-x_3$.  In the latter case the two $A$s are unitarily equivalent, by conjugation by $\diag(1,-1)$.  Hence we can replace ``There exists a ...'' by ``For every ...'' in (8) if we wish.\\
(iv)  Condition (9) of Theorem \ref{characE} furnishes a foliation of $\E$ by a family
of geodesic discs.  Indeed, for $\beta_1,\beta_2$ such that $|\beta_1|+|\beta_2| < 1$,
the map
\[
\ph_{\beta_1 \beta_2} : \D \to \E: \la \mapsto 
( \beta_1+\bar\beta_2 \la, \beta_2 + \bar\beta_1 \la, \la)
\]
satisfies $\Psi(\omega, .) \circ \ph_{\beta_1 \beta_2} \in \Aut \D$ 
for any $\omega \in \T$,  since $\Psi(\omega,.)$ is analytic from $\E$ to $\D$ and
\[
\Psi(\omega,  \ph_{\beta_1 \beta_2}(\la)) = c \frac{\alpha - \la}{1 - \bar\al\la}
\]
where
\[
c = \omega\frac{1-\bar\omega\bar\beta_2}{1-\omega\beta_2} \in \T, \quad
\al = \frac{\bar\omega\beta_1}{1-\bar\omega\bar\beta_2} \in \D.
\]
Since $\ph_{\beta_1 \beta_2}$ has a left inverse modulo $\Aut \D$ it is a complex geodesic of $\E$.  Since $\beta_1, \beta_2$ are determined, for any $x \in \E$, by equations
(\ref{defbetas}), each point of $\E$ lies on a unique disc $\ph_{\beta_1 \beta_2}(\D)$.
However, points of $\partial\E$ of the form $(x_1, \bar x_1 x_3, x_3)$ with $|x_3|=1$ lie on infinitely many discs $\ph_{\beta_1 \beta_2}(\Delta)$ (these are the points of the distinguished boundary of $\bar\E$;  see Theorem \ref{characdE} below).\\
(v)  Here is a geometric interpretation of the parameters $\beta_1, \beta_2$ in conditions (9).  For $x=(x_1,x_2,x_3) \in \bar\E$ let $\tilde{x}=(x_1,\bar x_3,\bar x_2)$.  As we have observed, $\tilde{x} \in \bar\E$.   In view of the formulae (\ref{imPsi}) and (\ref{defbetas}) we find that the disc $\Psi(\D, \tilde{x})$ has centre $\beta_1$ and radius $|\beta_2|$.   \hfill $\Box$
\end{Remark}

Note that any point $(x_1, x_2, x_3)$ of $\bar\E$ satisfies $x_1 x_2 = x_3$ if and only if any matrix representing it as in (7) of Theorem~\ref{closE} is either upper or lower triangular. This motivates the following definition:
\begin{Definition}\label{triangular}
We say that a point $(x_1, x_2, x_3)$ of $\bar\E$ is {\em triangular} if $x_1 x_2 = x_3$.
\end{Definition}

The characterization theorems show the close relation between $\E$ and two standard Cartan domains: the open unit balls $R_I(2,2), R_{II}(2)$ of the spaces of $2\times 2$ matrices and symmetric $2\times 2$ matrices respectively.  
Denote by $\pi$ the mapping
\begin{equation} \label{defpi}
\pi: \C^{2\times 2} \to \C^3: A=[a_{ij}] \mapsto (a_{11},a_{22},\det A).
\end{equation}
Two of the assertions of Theorem \ref{characE} are that $\E$ is the image under $\pi$ of both the Cartan domains $R_I(2,2)$ and $ R_{II}(2)$.  

Condition (2) of Theorem \ref{characE} shows that we can almost identify $\E$ with the space of M\"obius transformations that map $\Delta$ to $\D$ via the correspondence $x \mapsto \Psi(.,x)$.  For non-triangular $x$ (and equivalently non-constant $\Psi(.,x)$) this correspondence is bijective, but if $x$ is triangular then $\Psi(.,x)$ is the constant function equal to $x_1$, and so the whole disc $\{(x_1, \la, x_1\la): \la \in \D\} \subset \E$ maps to the same constant function $\Psi(.,x)$.  It is nevertheless often useful to think of $\E$ and $\bar\E$ as sets of M\"obius transformations.  In particular this viewpoint reveals a natural binary operation on $\E$, corresponding to the composition of M\"obius transformations.  We make this precise in Section \ref{autos}.
 
We conclude this section with some basic geometric properties of $\E$.  Firstly, both $\E$ and $\bar\E$ are non-convex: if $x= (1,i,i)$ and $y=(-i,1,-i)$ then $x,y \in \bar\E$ but 
$\tfrac 12 (x+y) \notin \bar\E$.  However, $\bar\E$ is contractible by virtue of the following result.
\begin{Theorem} \label{starlike}
$\E$ and $\bar\E$ are starlike about $(0,0,0)$ but are not circled.
\end{Theorem}
\begin{proof}
A straightforward verification shows that, for any $x\in\C^3, z\in\C$ and $ r > 0$,
\begin{equation}\label{ident}
|1-rzx_2|^2 - |rx_1 - rzx_3|^2 = r^2\{|1-zx_2|^2 - |x_1-zx_3|^2\} + (1-r)(1+r-2r\mathrm{Re}(zx_2)).
\end{equation}
Consider $x\in \bar\E, z\in \Delta$ and $0\leq r <1$.  By Theorem \ref{closE}, condition (2), we have $||\Psi(.,x)||_\infty \leq 1$ and hence $|1-zx_2|^2 - |x_1-zx_3|^2 \geq 0$.
It is also true that $1-r >0$ and $1+r-2r\mathrm{Re}(zx_2) > 0$.  It follows from the identity (\ref{ident}) that
\[
|1-rzx_2|^2 - |rx_1 - rzx_3|^2 > 0,
\]
or equivalently $\Psi(z, rx) \in \D$.  Thus $\Psi(.,rx)$ maps $\Delta$ into $\D$, and so $rx \in \E$.  Thus $\E$ and $\bar\E$ are starlike about $(0,0,0)$.  The point $x=(1,1,1)$ is in $\bar\E$ but $ix \notin\bar\E$, so that neither $\bar\E$ nor $\E$ is circled.
\end{proof}
Although $\E$ is not convex, $\E\cap \R^3$ {\em is}.
\begin{Theorem} \label{EcapR3}
$\E\cap \R^3$ is the open tetrahedron with vertices $(1,1,1), (1,-1,-1)$, \\ $ (-1,1,-1)$ and $(-1,-1,1)$.
\end{Theorem}
\begin{proof}
Let $x\in\R^3, |x_2| < 1$ and suppose $x$ non-triangular.  The centre of the disc $\Psi(\Delta,x)$ is real, to wit
$\frac{x_1-x_2x_3}{1-|x_2|^2}, $
and hence the point $\zeta$ of maximum modulus in $\Psi(\Delta,x)$ is also real.  It follows that $\Psi(.,x)^{-1}(\zeta) \in \R$, and so $\Psi(.,x)$ attains its maximum modulus over $\Delta$ at either $1$ or $-1$.  Hence $x\in\E$ if and only if
\[
 -1 <\Psi(-1,x) < 1 \mbox{  and  } -1 < \Psi(1,x) < 1,
\]
that is,
\begin{eqnarray} \label{4faces}
-x_1+x_2-x_3+1 > 0, &\quad& -x_1-x_2+x_3+1 > 0\\
x_1+x_2+x_3+1 >0, &\quad& x_1-x_2-x_3+ 1 >0. \nn
\end{eqnarray}
The four half-spaces defined by these inequalities intersect in the open tetrahedron with the four vertices in the statement of the theorem, and so $x\in\E$ if and only if $x$ lies in the tetrahedron.

If $|x_2|\geq 1$ then $x$ belongs neither to $\E$ nor to the tetrahedron.  If $x$ is triangular the inequalities (\ref{4faces}) reduce to
\begin{eqnarray*} 
(1-x_1)(1+x_2)> 0, &\quad& (1-x_1)(1-x_2) > 0\\
 (1+x_1)(1+x_2) >0, &\quad& (1+x_1)(1-x_2) >0,
\end{eqnarray*}
which is equivalent to $|x_1|< 1, |x_2|< 1$.  Thus in all cases, for $x\in \R^3$ we have $x\in\E$ if and only if $x$ lies in the tetrahedron.
\end{proof}
\begin{Theorem} \label{polyconvex}
$\bar\E$ is polynomially convex.
\end{Theorem}
\begin{proof}
Let $x\in \C^3\setminus\bar\E$.  We must find a polynomial $f$ such that $|f|\leq 1$ on $\bar\E$ and $|f(x)| > 1$.  If $x$ is triangular it suffices to take $f(x)=x_1$ or $f(x)=x_2$, and if any $|x_j| > 1$ we may take $f(x)=x_j$, so we assume that $x$ is non-triangular and $x \in \Delta^3$.  By Theorem \ref{closE}, condition (2), there exists $z\in\D$ such that $|\Psi(z,x)|> 1$, while $|\Psi(z,.)| \leq 1$ on $\bar\E$.  Let $f_N$ be the polynomial given by  $f_N(x)=(x_1-x_3z)(1+x_2z + x_2^2z^2+\dots+x_2^Nz^N)$; then, for any $y\in\Delta^3$,
\[
|f_N(y)-\Psi(z,y)| \leq \frac{2|z|^{N+1}}{1-|z|}.
\]
Let $0<\ep< \tfrac 13(|\Psi(z,x)|-1)$ and choose $N$ so large that $|f_N -\Psi(z,.)| < \ep$ on $\Delta^3$.  Then $|f_N| < 1+\ep$ on $\bar\E$ and $|f_N(x)| \geq 1+2\ep$.
Hence we can take $f=(1+\ep)^{-1}f_N$ .
\end{proof}
It follows that $\E$ is a domain of holomorphy (for example \cite[Theorem 3.4.2]{krantz}).  However, Theorem \ref{EcapR3} shows that $\E$ does not have a $C^1$ boundary, and consequently much of the theory of pseudoconvex domains does not apply to $\E$.

 \section{A Schwarz lemma for the tetrablock} \label{schwarz}
Criterion (7) of Theorem \ref{closE} tells us that $x \in \bar\E$ if and only if
$x=\pi(A)$ for some contractive $2\times 2$ matrix $A$.
It follows that any $2\times 2$ function $F$ in the Schur class determines an analytic function $\pi\circ F: \D \to \bar\E.$  The interpolation problem for $\bar\E$
can therefore be addressed with the aid of the rich classical interpolation theory of the Schur class: to prove Theorem \ref{schwarzL} we shall use a refinement of the Schur-Nevanlinna reduction process for which the following result will be useful.
\begin{Lemma}\label{anX}
Let $Z \in \C^{2\times 2}$ be such that $\|Z\| < 1$ and let $0\leq \rho < 1$.
Let
\begin{equation} \label{defM}
M(\rho) = \left[ \begin{array}{cc}
 [(1-\rho^2Z^*Z)(1-Z^*Z)^{-1}]_{11}  & [(1-\rho^2)(1-ZZ^*)^{-1}Z ]_{21} 
 \\
   {[(1-\rho^2) Z^* (1-ZZ^*)^{-1}]_{12}}  &  [(ZZ^*-\rho^2)(1-ZZ^*)^{-1}]_{22}
\end{array} \right].
\end{equation}
\begin{enumerate}
\item[\rm (1)]
There exists $X \in \C^{2\times 2}$ such that 
$\|X\| \leq \rho$ and $[\mathcal{M}_{-Z}(X)]_{22} =0$ if and only if $\det M(\rho) \leq 0$.  
\item[\rm (2)]
For any $2\times 2$ matrix $X$, $[\mathcal{M}_{-Z}(X)]_{22} =0$ if and only if there exists $\al \in \C^2 \setminus \{0\}$ such that  
\[
X^*u(\al) = v(\al)
\]
where
\begin{eqnarray} \label{defuv}
 u(\al) &=& (1-ZZ^*)^{-\tfrac 12} (\al_1Ze_1 +\al_2 e_2),  \\
 v(\al) &=&  -(1-Z^*Z)^{-\tfrac 12} (\al_1 e_1 + \al_2 Z^* e_2) \nn
\end{eqnarray}
and $e_1, e_2$ is the standard basis of $\C^2$.  
\item[\rm(3)] In particular, if $\det M(\rho) \leq 0$ then an
$X$ such that $||X|| \leq \rho$ and $[\mathcal{M}_{-Z}(X)]_{22} =0$
is given by
\[
X= \left\{ \begin{array}{cl} \displaystyle \frac{u(\al)v(\al)^*}{\|u(\al)\|^2} & \mbox{ if  } [Z]_{22} \neq 0 \\
	0 & \mbox{ if  } [Z]_{22} =0
\end{array} \right.
\]
for any $\al \in \C^2 \setminus \{0\}$ such that $\left<M(\rho)\al,\al \right> \leq 0$.
\end{enumerate}
\end{Lemma}
\begin{proof}
We may write
\[
\mathcal{M}_{-Z}(X) = (AX+B)(CX+D)^{-1}
\]
where
\[
\begin{array}{cc}
A = (1-ZZ^*)^{-\tfrac 12},  & B=(1-ZZ^*)^{-\tfrac 12}Z, \\
C = (1-Z^*Z)^{-\tfrac 12}Z^*, & D = (1-Z^*Z)^{-\tfrac 12}.
\end{array}
\]
With this notation equations (\ref{defuv}) become
\[
u(\al)=\al_1C^*e_1 + \al_2 A^* e_2, \quad v(\al)=-\al_1D^*e_1 -\al_2 B^*e_2.
\]
For any matrix $X$,
\begin{align} 
  [\mathcal{M}_{-Z}&(X)]_{22} =0  \nn \\
& \Leftrightarrow  \ip{(AX+B)(CX+D)^{-1} e_2}{ e_2 } =0   \nn\\
   { } &\Leftrightarrow  \mbox{ for some non-zero $\xi \in \C^2$ } 
\ip{(AX+B)\xi}{e_2} =0 \mbox{ and } (CX+D) \xi =e_2  \nn \\
   { } & \Leftrightarrow \mbox{ for some non-zero $\xi \in \C^2$ } \xi \perp
(X^*A^*+B^*)e_2 \mbox{ and } \xi \perp (X^*C^*+D^*)e_1 \nn\\
 { } & \Leftrightarrow \mbox{there exists } \al \in \C^2 \setminus \{0\} 
\mbox{ such that }  \nn\\
 &  \quad  \al_1(X^*C^*+D^*)e_1+ \al_2(X^*A^*+B^*)e_2 =0 \nn \\
{ } & \Leftrightarrow \mbox{there exists } \al \in \C^2 \setminus \{0\}
\mbox{ such that } \nn \\
    \nn &  \quad X^*(\al_1C^*e_1+\al_2A^*e_2) = -\al_1D^*e_1 - \al_2 B^*e_2\\
& \Leftrightarrow \mbox{there exists } \al \in \C^2 \setminus \{0\}
\mbox{ such that } X^*u(\al) = v(\al). \label{somealpha}
\end{align}
Hence statement (2) holds.
For any $\al$ there exists an $X$ such that  $X^*u(\al) = v(\al)$ and $\|X\| \leq \rho$ if and only if  $||v(\al)|| \leq \rho ||u(\al)||$.  Now
\begin{align}  \label{inuv}
 ||v(\al)||^2 - \rho^2 ||u(\al)||^2 &=\ip{(DD^*-\rho^2CC^*)e_1}{e_1}\al_1\bar\al_1 +
\ip{(BD^*-\rho^2AC^*)e_1}{e_1 } \al_1\bar\al_2 + \nn   \\
  &   \quad \ip{(DB^*-\rho^2CA^*)e_2}{e_1 }\al_2\bar\al_1 +
\ip{(BB^*-\rho^2AA^*)e_2}{e_2 } \al_2\bar\al_2  \nn \\
      &= \ip{ M(\rho)\al}{\al }.
\end{align}
Hence there exists an $X$ such that $||X|| \leq \rho$ and  $[\mathcal{M}_{-Z}(X)]_{22} =0$ if and only if $M(\rho)$ is not positive definite, that is, if and only if $\det M(\rho) \leq 0$, since it is easily seen that the $(1,1)$ entry of $M(\rho)$ is positive.  Statement (1) follows.

When $\det M(\rho) \leq 0$ we may find $\al \neq 0$ such that $\left< M(\rho)\al, \al\right> \leq 0$ and define $u(\al),v(\al)$ by equations (\ref{defuv}).  
Then $\|v(\al)\| \leq \rho\|u(\al)\|$.  If $u(\al)=0$ then also $v(\al)=0$ and equation (\ref{somealpha}) holds with $X=0$ and we have
\[
0=[\mathcal{M}(0)]_{22} = [Z]_{22}.
\]
If $[Z]_{22} \neq 0$ then $u(\al) \neq 0$ and an $X$ satisfying  the relations (\ref{somealpha}) and $\|X\| \leq \rho$ is $u(\al)v(\al)^*\|u(\al)\|^{-2}$.
\end{proof}
We denote by $B$ the Blaschke factor
\begin{equation} \label{defB}
B(\la) = \frac{\la_0 -\la}{1 - \bar\la_0 \la}.
\end{equation}
\begin{Lemma}\label{22zero}
Let $\la_0 \in \D \setminus \{0\}$, let $Z \in \C^{2\times 2}$ satisfy $\|Z\| <1$
and let $M(.)$ be given by equation \eqref{defM}.
\begin{enumerate}
\item[\rm (1)]
 There exists a function $G$ such that
\begin{equation} \label{propG}
  G \in \S_{2\times 2}, \quad [G(0)]_{22} = 0  \mbox{ and } G(\la_0) = Z
\end{equation}
if and only if $\det M(|\la_0|) \leq 0$.
\item[\rm (2)]
A function $G \in \S_{2\times 2}$ satisfies the conditions \eqref{propG}
if and only if there exists  $\al \in \C^2 \setminus\{0\}$ such that $\left< M(|\la_0|)\al,\al\right> \leq 0$ and a Schur function $Q$ such that $Q(0)^*\bar\la_0 u(\al) = v(\al)$
and
$G= \mathcal{M}_{-Z} \circ (BQ)$,
where
 $u(\al), v(\al)$ are given by equations \eqref{defuv}.
\item[\rm (3)]
In particular, if $[Z]_{22} \neq 0$ and $\al \in \C^2\setminus\{0\}$ satisfies
$\ip{ M(|\la_0|)\al}{\al} \leq 0$ then the function
\begin{equation} \label{Galpha}
G(\la) = \mathcal{M}_{-Z}\left( \frac{B(\la)u(\al)v(\al)^*}{\la_0 \|u(\al)\|^2} \right)
\end{equation}
satisfies the conditions \eqref{propG}.
 \end{enumerate}
\end{Lemma}
\begin{proof}
(2) If $G$ satisfies the conditions \eqref{propG}) then $\mathcal{M}_Z  \circ G \in \S_{2\times 2}$ vanishes at $\la_0$ and hence is of the form $BQ$ for some Schur function $Q$.  Then $G=\mathcal{M}_{-Z} \circ (BQ)$ and moreover
\[
[\mathcal{M}_{-Z}(\la_0 Q(0))]_{22} = 
[\mathcal{M}_{-Z} \circ (BQ)(0)]_{22} = [G(0)]_{22} = 0.
\]
Since $||\la_0 Q(0)|| \leq  |\la_0|$ Lemma \ref{anX} tells us that there exists $\al \neq 0$ such that $\left< M(|\la_0|)\al,\al \right> \leq 0$ and $(\la_0 Q(0))^* u(\al) = v(\al)$.  Thus necessity holds in statement (2).  The argument is reversible, and so (2) holds.

It is clear from statement (2) and equation (\ref{inuv}) that there is a $G$ satisfying conditions (\ref{propG}) if and only if $M(|\la_0|)$ is not positive definite.  Hence statement (1) holds.

(3) is also an easy consequence of (2), obtained by taking $Q$ to be the constant function whose value is the unique rank 1 matrix satisfying $Q^* \bar\la_0 u(\al)=v(\al)$
(as in Lemma \ref{anX}(3)).
\end{proof}
\begin{Remark}   \rm  
If $[Z]_{22} = 0$ then the constant function $G(\la)=Z$ has the desired properties.
\end{Remark}

\begin{Lemma} \label{ifsomethenall}
Let $\varphi: \D \to \bar\E$ be analytic.  If $\varphi$ maps some point of $\D$ into $\E$ then $\ph(\D) \subset \E$.
\end{Lemma}
\begin{proof}
Suppose that $\ph(\la_0)\in\E$ for some $\la_0\in\D$.  Since $\ph_2:\D\to\Delta$ is analytic and $|\ph_2(\la_0)| < 1$ it follows from the Schwarz-Pick Lemma that $\ph_2(\D)
\subset \D$.  Fix $z\in\Delta$.  The function $\la \mapsto \Psi(z,\ph(\la))$ is well defined and analytic on $\D$ and maps $\la_0$ into $\D$; hence it maps all of $\D$ into $\D$.
Now fix $\la\in\D$: the map $\Psi(.,\ph(\la)) $ maps $\Delta$ to $\D$, and hence, by Theorem \ref{characE}, $\ph(\la)\in\E$.
\end{proof}
We now prove the Schwarz Lemma for $\E$, the main result of the paper. \\
{\bf Proof of Theorem \ref{schwarzL}.}
It is clear from Lemma \ref{ifsomethenall} that (1)$\Leftrightarrow$(1$^\prime$).\\
 (1)$\Rightarrow$(2)  Let $\ph: \D \to \E$ be as in (1).  For any $\omega \in \T \quad \Psi(\omega,\ph(.))$ is an analytic self-map of $\D$ and
\[
\Psi(\omega, \ph(0)) = \Psi(\omega,0,0,0) = 0.
\]
By Schwarz' Lemma
\[
|\Psi(\omega,x)| = |\Psi(\omega,\ph(\la_0))| \leq |\la_0|.
\]
On taking the supremum over $\omega \in \T$ we find that $D(x) \leq |\la_0|$, that is,
\[
\frac{|a-\bar b p|+ |ab-p|}{1-|b|^2} \leq |\la_0|.
\]
By the same reasoning with $a,b$ interchanged we have
\[
\frac{|b-\bar a p|+|ab-p|}{1-|a|^2} \leq |\la_0|,
\]
and so (2) holds.\\
(2)$\Rightarrow$(3) is trivial.\\
(4)$\Rightarrow$(1)  If $F=[F_{ij}]$ is as in (4) then the function
\[
\ph = (F_{11}, F_{22}, \det F)
\]
is analytic in $\D$, satisfies
\[
\ph(0) = (0,0,0), \quad \ph(\la_0) = x
\]
and by condition (7) of Theorem \ref{closE}, satisfies $\ph(\D) \subset \bar\E$.\\
(3)$\Rightarrow$(4)  Suppose that $|b| \leq |a|$ and $D(x) \leq |\la_0|$.  Consider the case that $ab=p$.   Here $D(x) = |a|$, and so 
\[
|b| \leq |a| \leq |\la_0|.
\]
By Schwarz' Lemma there are analytic self-maps $f,g$ of $\D$ such that $f(0)=g(0)=0,
\quad f(\la_0) = a$ and $g(\la_0) = b$.  The function $F= \mathrm{diag} (f,g)$ 
then has the required properties, and so (3)$\Rightarrow$(4) when $ab=p$.

Now consider the case that $ab \neq p$. %and suppose to begin with that $D(x) < |\la_0|$.
  We shall construct $F \in \S_{2\times 2}$ such that
\[
F(0) = \left[ \begin{array}{cc} 0 & *\\ 0 & 0\end{array} \right], \quad
F(\la_0) = \left[ \begin{array}{cc} a &  w\\ \la_0 w & b\end{array} \right].
\]
where  $w$ is a square root of $(ab-p)/\la_0$.
It suffices to find $G \in \S_{2	\times 2}$ such that conditions (\ref{propG}) above hold for
\begin{equation} \label{defZ}
Z = \left[ \begin{array}{cc} a/\la_0 &  w \\ w & b
\end{array} \right],
\end{equation}
for then the function $F(\la)= G(\la)\mathrm{diag}(\la,1)$ has the required properties.  To obtain such a $G$ we shall invoke Lemma \ref{22zero}, which we can do provided that $||Z|| < 1$.

Let
\[
a'=a/\la_0, \quad b'= b/\la_0, \quad p' = p/\la_0.
\]
Since $D(x) \leq |\la_0|$ we have
\[
D(a',b,p') \leq 1 \mbox{ and } D(b',a,p') \leq 1,
\]
so that $(a',b,p'),  (a,b',p')  \in \bar\E$.  Hence
\[
|b| \leq |a| \leq |\la_0|, \quad |p| \leq |\la_0| < 1.
\]
It follows that $a \neq \bar b p$ and
\begin{equation} \label{Agt1}
|ab -p| < |a-\bar bp| + |ab-p| \leq |\la_0|(1-|b|^2).
\end{equation}
Moreover, since $(a, b',p') \in \bar\E$, condition (5) of Theorem \ref{closE} shows that
\[
1 - |a|^2 -|b'|^2 +|p'|^2 \geq \frac{2}{|\la_0|} |ab-p|.
\]
That is,
\begin{equation} \label{dfY1}
 2 \leq Y_1 \stackrel{\rm def}{=} 
 \frac{|\la_0|}{|ab-p|} \left( 1-|a|^2 - \frac{|b|^2}{|\la_0|^2} + \frac{|p|^2}{|\la_0|^2}\right)
\end{equation}
with strict inequality if and only if $(a,b',p') \in \E$, that is, if and only if 
$D(b,a,p) < |\la_0|$.
Likewise
\begin{equation} \label{dfY2}
 2 \leq Y_2 \stackrel{\rm def}{=} 
 \frac{|\la_0|}{|ab-p|} \left( 1-\frac{|a|^2 }{|\la_0|^2}- |b|^2+ \frac{|p|^2}{|\la_0|^2}\right),
\end{equation}
with strict inequality if and only if $D(a,b,p) < |\la_0|$.
\begin{Lemma} \label{normZlt1}
Let $x=(a,b,p) \in \E, \la_0 \in \D\setminus\{0\}$ and $w^2= (ab-p)/\la_0$.   Suppose that $|b| \leq |a|, \, D(x) \leq |\la_0|$ and $ab \neq p$.  Let $Z$ be defined by equation (\ref{defZ}).  Then  $||Z|| \leq 1$, with equality  if and only if $D(x)=|\la_0|$.
\end{Lemma}
\begin{proof}
We have
\[
1-Z^*Z = \left[ \begin{array}{cc} \ds
 1- \left|\frac{a}{\la_0}\right|^2 - |w|^2 & \ds
-\frac{\bar a w}{\bar\la_0} -b\bar w \\
 \ds -\frac{a\bar w}{\la_0} - \bar b w & \ds
1 -|b|^2 -|w|^2 \end{array} \right ],
\]
\[
\det(1-Z^*Z) = 1 - \left|\frac{a}{\la_0}\right|^2 -|b|^2 +
\left|\frac{p}{\la_0}\right|^2 -\frac{2|ab-p|}{|\la_0|}
\]
(see equation (\ref{formdet}) above).   Since $(a',b,p'), (a,b',p') \in \bar\E$, conditions ($3'$) and ($3$) respectively of Theorem \ref{closE} show that the diagonal entries of 
$1-Z^*Z$ are non-negative, while condition (5) of the same theorem shows that 
$\det(1-Z^*Z) \geq 0$.  Hence $1-Z^*Z \geq 0$.  By the corresponding conditions in Theorem \ref{characE}, the diagonal entries and determinant are all strictly positive if and only if $(a',b,p'), (a,b',p') \in \E$,  which is so if and only if $D(x) < |\la_0|$.

\end{proof}
To apply Lemma \ref{22zero} we need to know the sign of $\det M(|\la_0|).$  A routine (if laborious) calculation gives the following.
\begin{Lemma} \label{gotdetM}
Under the assumptions of Lemma {\rm \ref{normZlt1}}, if $D(x) < |\la_0|$ and $M(|\la_0|)$ is defined by equation \eqref{defM}  then
\begin{equation}\label {formM}
M(|\la_0|)\det(1-Z^*Z) =
\end{equation}
\[
  \left[ \begin{array}{cc}
1- |a|^2-|b|^2+|p|^2-|ab-p|\left( |\la_0|+\frac{1}{|\la_0|}\right) & \ds
(1-|\la_0|^2) \left( w+ \frac{p\bar w}{\la_0} \right) \\ \ds
(1-|\la_0|^2) \left(\bar w+ \frac{ \bar p w}{\bar\la_0}\right) & 
\begin{array}{c}
 -|\la_0|^2 +|a|^2+|b|^2- \left|\frac{p}{\la_0}\right|^2\\
+|ab-p|\left( |\la_0|+ \frac{1}{|\la_0|} \right) \end{array}
\end{array} \right]
\]
and
\begin{equation}\label {formdetM}
\det \left(M(|\la_0|)\det(1-Z^*Z)\right) = -(y-y_1)(y-y_2)
\end{equation}
where
\begin{eqnarray*}
 y &=& 2|ab-p| ,\\
y_1&=& |\la_0|\left( 1-|a|^2 -\left|\frac{b}{\la_0}\right|^2 + \left|\frac{p}{\la_0}\right|^2\right)=|ab-p|Y_1, \\
y_2 &=& |\la_0|\left( 1 -\left|\frac{a}{\la_0}\right|^2 -|b|^2+ \left|\frac{p}{\la_0}\right|^2\right)=|ab-p|Y_2. 
\end{eqnarray*} 
\end{Lemma}
We resume the proof that (3)$\Rightarrow$(4) when $ab \neq p$ and $|b|\leq |a|$.
Suppose first that $D(x) < |\la_0|$.
By Lemmas \ref{normZlt1} and \ref{gotdetM}
we have $||Z|| <1$ and
\[
\det(M(|\la_0|)\det(1-Z^*Z)) = - |ab-p|^2(2-Y_1)(2-Y_2) < 0,
\]
since $Y_1,Y_2 > 2$ by inequalities (\ref{dfY1}), (\ref{dfY2}).  Since $\det(1-Z^*Z) >0$ it follows that $\det M(|\la_0|) <0$.  By Lemma \ref{22zero} there exists 
$G \in \S_{2\times 2}$ such that $[G(0)]_{22} = 0$ and $G(\la_0) = Z$.  Let
\[
F(\la) = G(\la) \mathrm{diag}(\la, 1), \quad \la \in \D.
\]
Then $F \in \S_{2\times 2}$,
\begin{equation} \label{propF}
F(0) = \left[ \begin{array}{cc} 0 & * \\ 0 & 0  \end{array} \right] \mbox { and }
F(\la_0) = \left[ \begin{array}{cc} a &  w \\ \la_0 w & b \end{array} \right]
\end{equation}
where $w^2 = (ab-p)/\la_0$.  Since $\det F(\la_0)=p$ we have (3)$\Rightarrow$(4) in the case that $|b| \leq |a|$ and
$D(x) < |\la_0|.$  Similarly it holds if $|a| \leq |b|$ and $D(x) < |\la_0|$.

Now suppose that $D(x) = |\la_0|$.  Write $\la_\ep = \la_0(1+\ep)^2$ for $\ep >0$ so small that $|\la_\ep| < 1$.  Note that 
\[
\left(\frac{w}{1+\ep}\right)^2 = \frac{ab-p}{\la_\ep}.
\]
By the above reasoning there exists $F_\ep \in \S_{2\times 2}$ such that
\[
F_\ep(0) = \left[ \begin{array}{cc} 0 & * \\ 0 & 0  \end{array} \right] \mbox { and }
F_\ep(\la_0) = \left[ \begin{array}{cc} a & \frac {w}{1+\ep} \\ 
(1+\ep)\la_0 w & b \end{array} \right].
\]
By Montel's theorem some subsequence of $F_\ep$ converges uniformly on compact subsets
of $\D$ as $\ep \to 0$ to an analytic function $F$.  Clearly $F$ is in the Schur class and
satisfies equation (\ref{propF}).  Hence (3)$\Rightarrow$(4). \hfill $\Box$
%\end{Proof}
\begin{Corollary}\label{KandC}
For any $x=(a,b,p) \in \E$
\begin{eqnarray*}
\mathcal{C}_\E(0,x)&=&\mathcal{K}_\E(0,x)=\delta_\E(0,x) \\
    &=&	\max\left\{ \tanh^{-1}  \frac{|a-\bar b p|+ |ab-p|}{1-|b|^2}, 
    \tanh^{-1} \frac{|b-\bar a p|+|ab-p|}{1-|a|^2}        \right\}
\end{eqnarray*}
where $\mathcal{C}_\E, \mathcal{K}_\E$ and $\delta_\E$ are the Carath\'eodory distance, Kobayashi distance and Lempert functions of $\E$ respectively.
\end{Corollary}
For definitions of $\mathcal{C}_\E, \mathcal{K}_\E$ and $\delta_\E$ see for example \cite[Chapter 1]{JP}.
\begin{proof}
The equation
\[
\delta_\E(0,x) =
	\max\left\{ \tanh^{-1}  \frac{|a-\bar b p|+ |ab-p|}{1-|b|^2}, 
    \tanh^{-1} \frac{|b-\bar a p|+|ab-p|}{1-|a|^2}        \right\}
\]
is simply a re-statement of the equivalence (1$^\prime$)$\Leftrightarrow$(2) of Theorem \ref{schwarzL}.  By definition
\[
\mathcal{C}_\E(0,x)=\sup \tanh^{-1}|F(x)|
\]
over all analytic maps $F:\E\to\D$ such that $F(0)=0$.  On taking $F=\Psi(\omega,.),  \omega\in\T$, we find
\[
\mathcal{C}_\E(0,x) \geq \sup_{\omega\in\T} \tanh^{-1}|\Psi(\omega,x)| = \tanh^{-1}\frac{|a-\bar b p|+ |ab-p|}{1-|b|^2}
\]
and by symmetry
\[
\mathcal{C}_\E(0,x) \geq \max\left\{ \tanh^{-1}  \frac{|a-\bar b p|+ |ab-p|}{1-|b|^2}, 
    \tanh^{-1} \frac{|b-\bar a p|+|ab-p|}{1-|a|^2}        \right\} =\delta_\E(0,x).
\]
It is always true that
\[
\mathcal{C}_\E \leq\mathcal{K}_\E \leq\delta_\E;
\]
 the Corollary follows.
\end{proof}
\begin{Corollary} \label{whichD}
If $(a,b,p) \in \E$ and $|b| \leq |a|$ then
\[
\frac{|b-\bar a p|+|ab-p|}{1-|a|^2} \leq \frac{|a-\bar bp|+|ab-p|}{1-|b|^2}.
\]
\end{Corollary}
\begin{proof}
This follows from the implication (3)$\Rightarrow$(2) of Theorem \ref{schwarzL}
with $\la_0 = D(a,b,p)$.
\end{proof}

The proof of Theorem \ref{schwarzL} not only demonstrates the existence of an interpolating function $\ph:\D \to \E$ when $D(x) \leq |\la_0|$ but also shows us how to construct a suitable $\ph$.
\vspace*{0.2cm}
\begin{center} \bf Algorithm \end{center}

Let $x=(a,b,p) \in \E$, let $\la_0 \in \D \setminus \{0\}$ and suppose that
$|b| \leq |a|$ and
\[
\frac{|a-\bar bp|+|ab-p|}{1-|b|^2} < |\la_0|.
\]
An analytic function $\ph:\D \to \E$ such that $\ph(0)= (0,0,0)$ and $\ph(\la_0)=x$
can be found as follows.
\begin{enumerate}
\item[\rm 1.] If $b=0$ then let $\ph(\la) = \la x/\la_0$.  Otherwise:
\item[\rm 2.]  Choose $w$ such that $w^2 = (ab-p)/\la_0$ and let
\[
Z= \left[ \begin{array}{cc} a/\la_0 & w\\
 w & b \end{array} \right];
\]
then $||Z|| < 1.$
\item[\rm 3.]  Let $M(.)$ be defined by equation (\ref{defM}) and choose
$\al \in \C^2 \setminus \{0\}$ such that $\left< M(|\la_0|\al,\al \right> \leq 0$
(such an $\al$ exists by Lemmas \ref{normZlt1} and \ref{gotdetM}).
\item[\rm 4.] Let vectors $u, v \in \C^2$  and the Blaschke factor $B$ be given by equations (\ref{defuv}) and (\ref{defB}); note that by Lemma \ref{anX}, $u \neq 0$ since
$[Z]_{22} = b \neq 0$.
\item[\rm 5.] Let $F=[F_{ij}]$ be defined by
\begin{equation}\label{defF}
F(\la) = \mathcal{M}_{-Z} \left( \frac{B(\la)uv^*}{||u||^2}\right) \mathrm{diag}(\la,1);
\end{equation}
then $F \in \S_{2 \times 2}$.
\item[\rm 6.]  Let $\ph= (F_{11},F_{22}, \det F)$.  Then $\ph$ is analytic, maps  $\D$ to $\E$ and satisfies $\ph(0)=(0,0,0), \ph(\la_0) =x$.
\end{enumerate}
\begin{Remark}  \rm 
 (i) When $b = 0$ we have
\[
D(\la x/\la_0) = \left|\frac{\la}{\la_0}\right| D(x) \leq |\la|
\]
and so the simple recipe in Step 1 of the algorithm does indeed produce a mapping
$\ph$ such that $\ph(\D) \subset \E$.  In general, though, this recipe is insufficient.
Consider for example the point $x=(\tfrac 12,\tfrac 12,\tfrac 12) \in \E$: here
$D(x)= \tfrac 23$.  Take $\la_0$ slightly greater than $\tfrac 23$.  One can easily
check that $D(\mathrm{i} rx/\la_0) > 1$ for real $r$ close to $1$.  Hence $\la \mapsto \la x/\la_0$ does not map $\D$ into $\bar\E$.\\
(ii)  In the case that $D(x)=|\la_0|$ a modification of our construction will produce an interpolating function $\ph$: one uses the singular value decomposition of $Z$.  The construction is similar to the one in the next section; the details are left to the reader.
\end{Remark}

\section{Non-uniqueness in the Schwarz lemma} \label{uniq}
In contrast to Schwarz' original Lemma, there is no uniqueness statement in the case that the necessary and sufficient condition (2) of Theorem \ref{schwarzL} holds with equality.
Here is a numerical example.

Let $x=(a,b,p)=(\tfrac 12,\tfrac 14,\tfrac 12)$. We have, since $|b| \leq |a|$,
$$
\max\left\{ \frac{|a-\bar b p|+|ab-p|}{1-|b|^2}, \frac{|b- \bar a p|+|ab-p|}{1-|a|^2}  \right\}= \frac{|a-\bar b p|+|ab-p|}{1-|b|^2} = \tfrac 45.
$$
Let $\la_0 = -\tfrac 45$.  We shall construct infinitely many analytic $\ph: \D \to \bar\E$ such that
$\ph(0)=(0,0,0)$ and $ \ph(\la_0)= x$.
 
Let $ w^2= \frac{ab-p}{\la_0} =\frac{15}{32}$   and
$$
Z=\left[\begin{array}{cc}  a/\la_0 & w\\ w& b \end{array}\right].
$$
Then $||Z|| =1$ and since $Z$ is Hermitian we may diagonalise $Z$ as follows:
$$
Z= \left[
\begin{array}{cc} -\tfrac 58 & w\\w & \tfrac 14 \end{array}\right] =
U^*\left[ \begin{array}{cc} -1 & 0\\ 0 & \tfrac 58 \end{array} \right] U,
$$
where $U$ is the unitary matrix
$$
U=  \left[ \begin{array}{cc} 8w & 4w \\ -3 & 5 \end{array}\right] \diag\left(\frac{1}{\sqrt{39}}, \sqrt{\frac{2}{65}}\right).
$$
If $G$ is a Schur function such that $G(\la_0)=Z$ then
$U^*GU$ is a Schur function whose value at $\la_0$ is $\diag(-1, \tfrac 58)$,
from which it is clear that $U^*GU=\diag(-1,g)$ for some scalar function $g$ in the Schur class satisfying $g(\la_0)= \tfrac 58$.  We then have
$$
[G(0)]_{22} = \frac{1}{13}(10g(0)-3).
$$
It follows that the set of functions $G$ in the Schur class such that $G(\la_0)=Z, \quad
[G(0)]_{22} =0$ consists precisely of the functions $U\diag(-1,g)U^*$ where $g$ is a function in the Schur class such that $g(0)=3/10$ and $g(-4/5) = 5/8$.  There are infinitely many such $g$, since the pseudohyperbolic distance $d(\tfrac{3}{10}, \frac{5}{8}) = \tfrac 25 <  \tfrac 45=|\la_0|.$  For any such $g$ we define $\ph:\D \to \bar\E$
by $\ph= (F_{11}, F_{22}, \det F)$ where $F(\la)=U\diag(-1,g)U^*\diag(\la,1)$.  Note that $\ph_3(\la)=  -\la g(\la)$, so that distinct $g$ give rise to different mappings $\ph$, all analytic and satisfying $\ph(0)=(0,0,0), \ph(\la_0)=x$.

\section{All interpolating functions} \label{all}
The algorithm in Section \ref{schwarz} produces a single analytic function $\ph:\D \to \E$ satisfying a pair of interpolation conditions.  Our method of proof in fact gives more: a description of {\em all} such functions.  In an engineering context one could use the freedom in the solution to meet further performance specifications. 
\begin{Theorem}
Let $x=(a,b,p) \in \E$ and $ \la_0 \in \D \setminus \{0\}$ and suppose that $ab \neq p,
|b| \leq |a|$ and
$D(x) < |\la_0|$.   The set $\mathcal{I}$ of analytic functions $\ph:\D \to \E$ such that $\ph(0)=(0,0,0)$ and $\ph(\la_0)=x$ can be described as follows.

Let $w^2 = (ab-p)/\la_0$ and let $\xi_1,\xi_2$ be the roots of the equation $\xi+1/\xi =Y_2$ where $Y_2$ is defined by equation \eqref{dfY2}.  For any $\sigma >0$ let
\[
Z(\sigma) = \left[ \begin{array}{cc} a/\la_0 & \sigma w \\ \sigma^{-1} w & b \end{array} \right]
\]
and let $M(.)$ be defined by equation \eqref{defM} with $Z=Z(\sigma)$.  For any $\sigma$ such that 
\begin{equation} \label{cndsig}
\xi_1 < \sigma^2 < \xi_2
\end{equation}
we have $||Z(\sigma)|| < 1$ and $M(|\la_0|)$  is not positive definite.  Furthermore, for any $\al \in \C^2 \setminus \{0\}$ such that 
\begin{equation} \label{cndal}
\left< M(|\la_0|)\al,\al \right> \leq 0
\end{equation} 
and any $2\times 2$ function $Q$ in the 
Schur class such that 
\begin{equation} \label{cndQ}
Q(0)^*\bar\la_0u(\al)=v(\al),
\end{equation}
 where $u(\al), v(\al)$ are given by equations \eqref{defuv}, the function $ \pi \circ F$  belongs to $\mathcal{I}$ where 
\begin{equation} \label{gotF}
F(\la)= \mathcal{M}_{-Z(\sigma)}\circ (BQ)(\la) \mathrm{diag}(\la,1).
\end{equation}
  Conversely, every function in $\mathcal{I}$ is of the form $\pi \circ F$ for some choice of $\sigma, \al$ and $Q$ satisfying the  conditions \eqref{cndsig}, \eqref{cndal} and \eqref{cndQ} respectively and for $F$ given by equation \eqref{gotF}.
\end{Theorem}  
\begin{proof}
A slight modification of the proof of Lemma \ref{normZlt1} shows that $||Z(\sigma)|| < 1$ if and only if 
\[
1 - |b|^2 -\sigma^2\left|\frac{ab-p}{\la_0}\right| >0 \mbox{ and }
1 - \left|\frac{a}{\la_0}\right|^2 -|b|^2+\left|\frac{p}{\la_0}\right|^2
 - \left|\frac{ab-p}{\la_0}\right| (\sigma^2 +\frac{1}{\sigma^2}) >0,
\]
that is, if and only if
\[
\sigma^2 < K \mbox{  and  } \sigma^2 + \frac{1}{\sigma^2} < Y_2
\]
where 
\begin{equation} \label{defK}
K \stackrel{\rm def}{=} \frac{|\la_0|(1-|b|^2)}{|ab-p|}.
\end{equation}
By the inequalities (\ref{Agt1}) and (\ref{dfY2}),
\begin{equation} \label{KY1}
K > 1 \mbox{ and } Y_2 > 2.
\end{equation}
Moreover
\begin{equation} \label{KY2}
K + \frac{1}{K} - Y_2 = \frac {|a-\bar b p|^2}{|\la_0|(1-|b|^2)|ab-p|} > 0.
\end{equation}
Figure 1 is a plot of $\xi+1/\xi$ against $\xi$ that incorporates the relations (\ref{KY1}) and (\ref{KY2}).  It is clear from the plot that $\xi < K$ and $\xi+1/\xi < Y_2$ precisely when $\xi_1 < \xi < \xi_2$, or 
\begin{center} \includegraphics {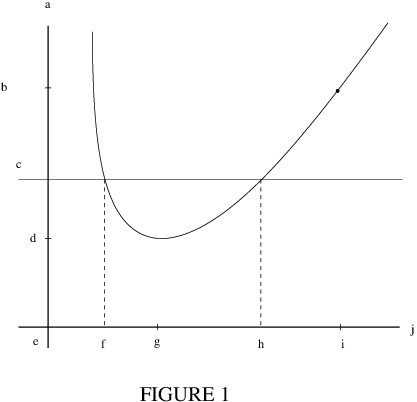} \end{center}
equivalently, when $\xi+1/\xi < Y_2$, the inequality 
$\xi < K$ then being automatically satisfied.  
\begin{comment}
\begin{figure}[!ht]
\begin{center}
\psfragscanon
\psfrag{a}[][][1.0][0]{\footnotesize{$\xi+\displaystyle\frac{1}{\xi}$}}
\psfrag{b}[][][1.0][0]{\footnotesize{$K+\displaystyle\frac{1}{K}$}}
\psfrag{c}[][][1.0][0]{\footnotesize{$Y_2$}}
\psfrag{d}[][][1.0][0]{$2$}
\psfrag{e}[][][1.0][0]{$0$}
\psfrag{f}[][][1.0][0]{$\xi_1$}
\psfrag{g}[][][1.0][0]{$1$}
\psfrag{h}[][][1.0][0]{$\xi_2$}
\psfrag{i}[][][1.0][0]{\footnotesize{$K$}}
\psfrag{j}[][][1.0][0]{$\xi$}
\includegraphics[width=0.4\textwidth]{fig1_sch.eps}\\
%\caption{Figure 1} 
\end{comment}
Figure 1
%\end{center}
%\end{figure}
It follows that $||Z(\sigma)|| < 1$ if and only if $\xi_1 < \sigma^2 < \xi_2$.

We claim that, for the same range of values of $\sigma$, $\det M(|\la_0|) < 0$.
Indeed, a straightforward calculation gives
\[
\det(M(|\la_0|)\det(1-Z(\sigma)^*Z(\sigma))) = -(y-y_1)(y-y_2)
\]
where
\begin{eqnarray*}
 y &=& |ab-p| (\sigma^2 + \frac{1}{\sigma^2}), \\
 y_j &=& |ab-p| Y_j, \quad j=1,2.\\
\end{eqnarray*}
Since
\[
Y_1-Y_2 = \frac{1-|\la_0|^2}{|(ab-p)\la_0|} (|a|^2 - |b|^2) \geq 0,
\]
it is clear that $\det M(|\la_0|) < 0$ when $y<y_2$, or equivalently, when $\sigma^2 + \frac{1}{\sigma^2}  < Y_2$, which is to say, when $\xi_1 < \sigma^2 < \xi_2$.
Thus $||Z(\sigma)|| < 1$ and $M(|\la_0|)$ is not positive definite when condition (\ref{cndsig}) is satisfied.

Suppose that $\sigma, \al$ and $Q$ satisfy conditions (\ref{cndsig}), (\ref{cndal}) and (\ref{cndQ}).  By Lemma \ref{22zero} the function $G =\mathcal{M}_{Z(\sigma)} \circ (BQ)$ belongs to $\S_{2\times 2}$ and satisfies $[G(0)]_{22} =0$ and $G(\la_0) = Z(\sigma)$.  Hence $F$ given by equation (\ref{gotF}) satisfies
\[
 F \in \S_{2\times 2}, \quad F(0) = \left[ \begin{array}{cc} 0 & *\\ 0 & 0 \end{array} \right] \mbox{ and } F(\la_0) = \left[ \begin{array}{cc} a & \sigma w \\ \la_0 \sigma^{-1} w & b \end{array} \right].
\]
Thus the function $\ph = \pi \circ F$ is analytic from $\D$ to $\E$ and satisfies $\ph(0)=(0,0,0)$ and $\ph(\la_0) = (a,b,p) = x.$ Thus $\ph \in \mathcal{I}$.

Conversely, suppose that $\ph \in \mathcal{I}$.  The radial limit function of $\ph$, which we shall again denote by $\ph$, maps $\T$ almost everywhere to $\bar\E$.   By a theorem of F. Riesz, or directly from inner-outer factorization, there exist $f,g \in H^\infty$ such that $fg = \ph_1\ph_2-\ph_3$ and $|f| = |g|$  a.e. on $\T$.  Since $fg(0)= 0$ we can assume that $g(0)=0$.  Let 
\begin{equation} \label{defineF}
F=\left[ \begin{array}{cc} \ph_1 & f \\ g & \ph_2 \end{array} \right].
\end{equation}
We have $\pi \circ F = \ph$.
By Lemma \ref{AstarA}, $1-F^*F$ has diagonal entries $1 - |\ph_1|^2 - |\ph_1\ph_2 - \ph_3|$ and $1-|\ph_2|^2 - |\ph_1\ph_2-\ph_3|$ and determinant $1- |\ph_1|^2 -|\ph_2|^2+|\ph_3|^2 -2|\ph_1\ph_2 - \ph_3|$  a.e. on $\T$,  and since these three functions are non-negative by Theorem \ref{closE}, it follows that $F$ is in the Schur class.  From the facts that $ab \neq p$, $\pi\circ F(\la_0)=(a,b,p)$  and
\[
F(0) = \left[ \begin{array}{cc} 0& f(0) \\ 0 & 0 \end{array} \right]
\]
one sees that $F$ is non-constant and hence $|f(0)| < 1$.  Thus $||F(0)|| < 1$ and so
 $F \in \S_{2\times 2}$.  We shall show that $F$ can be written in the form (\ref{gotF}) for some choice of $\sigma, \al$ and $Q$.

Note that $(fg)(\la_0) = ab-p \neq 0$, so that $f(\la_0), g(\la_0)$ are nonzero.
Let $\sigma = f(\la_0)/w$; then $g(\la_0)= \la_0 \sigma^{-1} w$.  We can suppose that $\sigma > 0$ (if necessary replace $F$ by $U^*FU$ for some constant diagonal unitary $U$).  Thus
\[
F(\la_0) = \left[ \begin{array}{cc} a & \sigma w \\ \la_0 \sigma^{-1} w & b \end{array} \right].
\]
Since  the first column of $F(0)$ is zero we may write $F(\la)= G(\la) \mathrm{diag}(\la,1)$ for some $G \in \S_{2\times 2}$.  We have
\[
G(0) = \left[ \begin{array}{cc} * & * \\ * & 0\end{array}\right], \quad
 G(\la_0) = \left[ \begin{array}{cc} a/\la_0 & \sigma w \\ \sigma^{-1} w & b \end{array} \right] = Z(\sigma).
\]
Since $||G(\la_0)|| < 1$ it follows that  condition (\ref{cndsig}) holds and thence that
$M(|\la_0|)$ is not positive definite.  By Lemma \ref{22zero}(2) there exist $\al, Q$ such that conditions (\ref{cndal}), (\ref{cndQ}) hold and $F$ is given by equation (\ref{gotF}).
\end{proof}
\begin{Remark} \rm
 The $2\times 2$ functions $Q$ in the Schur class that satisfy condition (\ref{cndQ}) can easily be parametrised by standard Nevanlinna-Pick theory (see for example \cite[Theorem 18.5.2 and Example 18.5.2]{BGR}).
\end{Remark}

\section{Automorphisms of the tetrablock} \label{autos}
In this section we shall use ``composition" on $\E$ to describe a large group of automorphisms of $\E$.  

Let $x, y \in \E$.  A simple calculation shows that
\[
\Psi(.,x) \circ \Psi(.,y) = \Psi(., x \diamond y)
\]
where
\begin{eqnarray} \label{defdia}
x\dia y &=& \frac{1}{1-x_2y_1} ( x_1 - x_3y_1, y_2 - x_2y_3, x_1y_2 - x_3y_3)\nonumber \\
  &=& \left( \Psi(y_1,x), \Upsilon(x_2,y), \frac{x_1y_2 - x_3y_3}{1-x_2y_1}\right).
\end{eqnarray}
Note that $1-x_2y_1 \neq 0$ since $|x_2| < 1, |y_1| < 1$, and hence $x\dia y$ is defined.
We shall define $x\dia y$ by equation (\ref{defdia}) for any $x, y \in \C^3$ such that $x_2 y_1 \neq 1$.   For $x,y \in \bar\E$, $x\dia y$ can fail to be defined, but if it {\em is} defined then $\Psi(.,x\dia y)$ is a self map of $\Delta$ and so $x\dia y \in \bar\E$ (in the triangular case we do have $|\Psi(y_1,x)| \leq 1, |\Upsilon(x_2,y)| \leq 1$ by virtue of Theorem \ref{closE}, conditions (2) and ($2'$)).

We can think of $\dia$ as a disguised form of matrix multiplication.  
For $x \in \bar\E$ let 
\[
M_x\stackrel{\rm def}{=}\left[\begin{array}{cc} x_3 & -x_1\\x_2 & -1 \end{array}\right].
\]
In the customary association of M\"obius transformations with $2\times 2$ matrices $\Psi(.,x)$ corresponds to the non-zero multiples of $M_x$  and so $\Psi(.,x\dia y)$ corresponds to
\[
M_{x\dia y}=
\la \left[\begin{array}{cc} x_3 & -x_1\\x_2 & -1 \end{array}\right] 
\left[\begin{array}{cc} y_3 & -y_1\\y_2 & -1 \end{array}\right] = \la
\left[\begin{array}{cc} x_3y_3-x_1y_2 & -x_3y_1+x_1\\x_2y_3-y_2 & -x_2y_1+1 \end{array}\right]
\]
where $\la$ is chosen to make the $(1,1)$ entry of the product equal to $-1$.

It follows from the associativity of matrix multiplication that $\dia$ is associative: for $u,v, w\in \bar\E$,  $(u\dia v)\dia w = u\dia(v\dia w)$ provided both sides are defined, for both sides have representing matrices proportional to $M_u M_v M_w$. 

We define left and right actions of $\Aut\D$ on both $\E$ and $\bar\E$.  
Let us write
\begin{equation} \label{defEsharp}
\E^\sharp \stackrel{\rm def}{=}\{x\in \bar\E: |x_1|<1,|x_2| <1 \},
\end{equation}
so that $\E\subset \E^\sharp \subset \bar\E$.
It is clear from equations (\ref{defdia}) that
$x\dia y$ is defined if $x, y \in \bar\E$ and one of $x,y$ lies in $\E^\sharp$, and moreover that $\E^\sharp$ is closed under the operation $\dia$.
Thus $(\E^\sharp,\dia)$ is a semigroup, with identity $(0,0,-1)$.  It contains $\Aut\D$ in a sense we now explain.\\
%{\color{blue}  It is not true that $x\in \E^\sharp, y\in \bar\E$ imply $x\dia y \in \E^\sharp$: take $y=(1,1,1)$.} \\
Consider any $\up \in \Aut\D$.  We can write
\begin{equation} \label{formup}
\up(z)=  \omega \frac{ z - \alpha}{\bar\alpha z - 1}= \Psi(z, \omega\al,\bar\al,\omega)
  \end{equation}
  for some  $ \alpha \in \D $ and $  \omega \in \T$.   Let
\begin{equation} \label{deftau}
\tau(\up) = (\omega\al,\bar\al,\omega),
\end{equation}
so that $\up=\Psi(.,\tau(\up))$.  Clearly $\tau(\up)$ is non-triangular, and since $||\up||_\infty = 1$ it follows from Theorem \ref{closE} that $\tau(\up) \in \bar\E$.    The first two components of $\tau(\up)$ have modulus less than one, so that $\tau(\up) \in \E^\sharp$.   
\begin{Lemma} \label{tauhom}
$\tau:\Aut \D \to \E^\sharp$ is a unital monomorphism of semigroups.
\end{Lemma} 
\begin{proof}
For $\up, \chi \in \Aut\D$ we have
\[
\Psi(.,\tau(\up\circ\chi))=\up\circ\chi=\Psi(.,\tau(\up))\circ\Psi(.,\tau(\chi))=\Psi(.,\tau(\up)\dia\tau(\chi)).
\]
Hence $\tau(\up\circ\chi)=\tau(\up)\dia\tau(\chi)$.  If $\iota$ is the identity automorphism on $\D$, then by equation \ref{deftau} we have $\tau(\iota)= (0,0,-1)$, which is the identity element of $\E^\sharp$.  It is clear that $\tau$ is injective, and so $\tau$ is a unital monomorphism.
\end{proof}
Henceforth we write $\up\cdot x$ for $\tau(\up)\dia x$ and $x\cdot \up$ for $x\dia \tau(\up)$. 
\begin{Lemma} \label{LactionE}
 For any automorphism $\up$ of $\D$ and any $x \in \bar\E$ we have
$\up\cdot x \in \bar\E$ and $ x\cdot\up \in \bar\E$.
 Moreover, if $x$ is in $\E$ then so are $\up\cdot x$ and $x\cdot \up$. 
 \end{Lemma}
\begin{proof}
Since $x\in\bar\E$ and $\tau(\up)\in\E^\sharp$ it follows that $\tau(\up)\dia x$ exists and belongs to $\bar\E$.  Likewise $x\dia\tau(\up) \in \bar\E$.  If further $x \in \E$ then $\Psi(.,x)$ maps $\Delta$ into $\D$, and since
\begin{equation} \label{compup}
\Psi(.,\up\cdot x) = \Psi(.,\tau(\up)\dia x) = \up \circ \Psi(.,x),
\end{equation}
it follows that $\Psi(.,\up\cdot x)$ also maps $\Delta$ into $\D$.
Now $\up\cdot x$ is triangular if and only if $\up \circ \Psi(.,x)$ is constant, which is so if and only if $x$ is triangular.  Hence, by Theorem \ref{characE}, if $x$ is non-triangular, $\up\cdot x$  lies in $\E$.  In the case of triangular $x$ the same conclusions hold:  here $\up\cdot x$ is  triangular, and in view of Theorem \ref{closE}, Condition (2), we need also to check that the second component of $\up\cdot x$ lies in $\D$.  By equations (\ref{defdia}), this component is $\Upsilon(\bar\al, x)$, which equals $x_2$ and does lie in $\D$.  Likewise if $x \in\E$ then  $x\cdot \up$ lies in $\E$.
\end{proof}
 Accordingly there are maps
\[
m_1:\bar\E \times \Aut\D \to \bar\E: (x,\up) \mapsto x\cdot \up, \quad
   m_2:(\Aut\D) \times \bar\E \to \bar\E:(\up, x) \mapsto \up\cdot x,
\]
 which restrict to maps $ \E \times \Aut\D \to \E$ and
    $\Aut\D\times \E \to \E$. 
  \begin{Theorem} \label{actionE}
 The maps $m_1$ and $m_2$ define right and left group actions of $\Aut\D$ on
  $\bar\E$ (and by restriction on $\E$) which commute with each other. Moreover the actions on $\E$ and $\bar\E$ are by maps that are holomorphic in a neighbourhood of $\bar\E$. 
 \end{Theorem} 
 \begin{proof}
 It follows from equations (\ref{defdia}) that $\iota\cdot x =(0,0,-1)\dia x =x$, and similarly $x\cdot\iota = x$ for any $x\in\bar\E$.   From the homomorphic property of $\tau$ and the associativity of $\dia$ we have
\[
\up\cdot(\chi\cdot x) =\tau(\up)\dia(\tau(\chi)\dia x)=(\tau(\up)\dia\tau(\chi))\dia x =\tau(\up\circ \chi)\dia x = (\up\circ\chi)\cdot x.
\]
Thus $m_2$ is a left action of $\Aut\D$ on $\E$ and $\bar\E$.  Similarly $m_1$ is a right action.

We must show that the left and right actions commute, that is, that $\up\cdot(x\cdot\chi) =(\up\cdot x)\cdot\chi$ for $\up,\chi, x$ as above.  This also follows from the associativity of the operation $\dia$.

Finally, the actions of $\Aut\D$ on $\E$ are given by rational functions: if $\up$ is given by equation (\ref{formup}) then
\begin{eqnarray*}
\up\cdot x &=& \tau(\up)\dia x= (\omega\al,\bar\al,\omega) \dia x \\
	&=& \frac{1}{1-\bar\al x_1}(\omega(\al-x_1), x_2-\bar\al x_3, \omega(\al x_2-x_3)).
\end{eqnarray*}
For fixed $\up \in \Aut\D$ this is clearly an analytic function of $x$ in the set $\{x\in \C^3: |x_1| < 1/|\al|\}$, which is a neighbourhood of $\bar\E$.
\end{proof}

It follows from Theorem \ref{actionE} that, for $\up,\chi \in \Aut\D$, there are commuting elements $L_\up, R_\chi \in \Aut\E$ given by $L_\up=m_2(\up,.),R_\chi=m_1(.,\up)$.
Another automorphism of $\E$ is the ``flip'' $F$:
\[
F(x_1,x_2,x_3)= (x_2,x_1,x_3), \quad x \in \E^\sharp.
\]
One can verify from equations (\ref{defdia}) that
\[
F(x\dia y) = F(y) \dia F(x), \quad x,y \in \E^\sharp.
\]
Moreover
\[
F(\tau(\up))=F(\omega\al,\bar\al,\omega)=(\al,\omega\al,\omega)=\tau(\up_*)
\]
where $\up_* \in \Aut\D$,
\[
\up_*(z) = \omega \frac{z-\bar\omega\bar\al}{\omega\al z -1}.
\]
\begin{Theorem} \label{agroup}
The set 
\[
G=\{ L_\up R_\chi F^\nu: \up,\chi \in \Aut\D, \nu = 0 \mbox{ or } 1\}
\]
constitutes a group of automorphisms of $\E$.
\end{Theorem}
\begin{proof}
It is clear that $G \subset \Aut\E$.  We need relations between the generators.
For $x\in\E, \up\in\Aut\D$,
\begin{eqnarray*}
FL_\up(x) &=& F(\up\cdot x) =F(\tau(\up)\dia x) = F(x)\dia F(\tau(\up)) = F(x)\dia\tau(\up_*) = F(x)\cdot \up_*\\
 &=& R_{\up_*}F(x).
\end{eqnarray*}
Similarly $FR_\up =L_{\up_*}F$.   It follows easily that $G$ is closed under the group operation and inversion,
hence is a subgroup of $\Aut\E$.
\end{proof}
We propose the following natural
\begin{Conjecture}\label{allautos}
$G = \Aut\E$:  every automorphism of $\E$ is of the form $L_\up R_\chi F^\nu$ for some $ \up,\chi \in \Aut\D$ and $ \nu = 0 \mbox{ or } 1.$
 \end{Conjecture} 
\begin{Remark}\label{orbitof0} \rm
The orbit of $(0,0,0)$ under $G$ is the set $\mathcal{T}$  of triangular points in $\E$.  Observe that $L_\up, R_\chi$ and $F$ all leave $\mathcal{T}$  invariant  (if $\Psi(.,x)$ is constant then so is $\up\circ\Psi(.,x)$), and so $\mathcal{T}$ is invariant under $G$.  Moreover $G$ acts transitively on $\mathcal{T}$, as the following lemma shows.
\end{Remark}
\begin{Lemma}\label{transonT}
If $x$ is a triangular point of $\E$ then $\up\cdot x \cdot\chi=(0,0,0)$ where $\up, \chi \in \Aut\D$ are given by
\begin{equation}\label{upchi}
\up(z)=\frac{z-x_1}{\bar x_1 z-1} \mbox{  and  } \chi(z) = \frac{z+\bar x_2}{x_2z+1}.
\end{equation}
\end{Lemma}
\begin{proof}
Let $M(\chi),M(\up)$ be the $2\times 2$ matrices corresponding to $\chi, \up$ respectively.   $M_x$ has rank one and
\begin{eqnarray*}
M_{\up\cdot x \cdot\chi}&=&\la M(\up) M_x M(\chi) =
	\la \left[\begin{array}{cc} 1 & -x_1 \\ \bar x_1 & -1\end{array}\right]
	\left[\begin{array}{c} x_1 \\ 1 \end{array} \right] \left[\begin{array}{cc} x_2 & -1\end{array}\right]
	\left[\begin{array}{cc} 1 & \bar x_2 \\  x_2 & 1\end{array}\right] \\
	&=&	\la\left[\begin{array}{cc} 0 & 0 \\ 0  & (1-|x_1|^2)(1-|x_2|^2) \end{array}\right]
\end{eqnarray*}
for some non-zero $\la$.  It follows that $\up\cdot x \cdot\chi=(0,0,0)$.
\end{proof}
By combining this lemma with the Schwarz Lemma, Theorem \ref{schwarzL}, we can obtain an explicit necessary and sufficient condition for the existence of an analytic map $\ph:\D \to \bar\E$ mapping any given pair of points in $\D$ to a given pair of points $x,y \in \E$ of which one is triangular.  
\begin{Corollary} \label{schwpick}
Let $x,y \in \E$, let $\la_1,\la_2$ be distinct points of  $\D$ and suppose that $x_1x_2=x_3$.  There exists an analytic map $\ph: \D \to \bar\E$ such that $\ph(\la_1) = x$ and 
$\ph(\la_2) = y$ if and only if
\begin{eqnarray*}
  &\max &\left\{  \frac{(1-|x_1|^2)|y_3-y_1y_2|+|y_1-\bar y_2y_3-x_1(1+|y_1|^2-|y_2|^2-|y_3|^2)+x_1^2(\bar{y}_1-y_2\bar{y}_3)|}{|1-\bar x_1y_1|^2-|y_2-\bar x_1y_3|^2} , \right. \\
     &&\left.  \frac{(1-|x_2|^2)|y_3-y_1y_2|+|y_2-\bar y_1y_3-x_2(1-|y_1|^2+|y_2|^2-|y_3|^2)+x_2^2(\bar y_2-y_1\bar y_3)|}{|1-\bar{x}_2y_2|^2-|y_1-\bar{x}_2y_3|^2}   \right\}\\
 {} &&\leq \quad |d(\la_1,\la_2)|
\end{eqnarray*}
where $d$ denotes the pseudohyperbolic distance on $\D$.
\end{Corollary}
\begin{proof}
Let $\up,\chi$ be given by equations (\ref{upchi}), so that $\up\cdot x \cdot\chi =(0,0,0), $ and let $y'= \up\cdot y \cdot\chi $.
 Some automorphism of $\D$ maps $\la_1,\la_2$ to $0,d(\la_1,\la_2)$, and so, by Theorem \ref{schwarzL}, the required map $\ph$ exists if and only if
\begin{equation} \label {criter}
\max \left\{ \frac{|y'_1-\bar y'_2y'_3|+|y'_1y'_2-y'_3|}{1-|y'_2|^2}, \frac{|y'_2-\bar y'_1y'_3|+|y'_1y'_2-y'_3|}{1-|y'_1|^2}\right\}
 \leq d(\la_1,\la_2).
\end{equation}
We have
\[
M_{y'}=\zeta \left[\begin{array}{cc} 1 & -x_1 \\ \bar x_1 & -1 \end{array}\right] M_y  
\left[\begin{array}{cc} 1 & \bar x_2 \\ x_2 & 1 \end{array}\right] 
\]
for some non-zero $\zeta$.  Hence
\[
y'_1y'_2-y'_3 = \det M_{y'} = -\zeta^2 (1-|x_1|^2)(1-|x_2|^2)(y_1y_2-y_3).
\]
Furthermore, if $J=\mathrm{diag(-1,1)}$,
\[
M_y J M_y^* = \left[ \begin{array}{cc} |y_1|^2-|y_3|^2 & y_1-\bar y_2y_3 \\ \bar y_1-y_2\bar y_3 & 1-|y_2|^2 \end{array} \right].
\]
Since
\[
  \left[ \begin{array}{cc} 1 & \bar x_2 \\ x_2 & 1 \end{array}\right]  J  \left[ \begin{array}{cc} 1 & \bar x_2 \\ x_2 & 1 \end{array}\right] = (1-|x_2|^2)J,
\]
we have
\[
M_{y'}JM_{y'}^*=|\zeta|^2 (1-|x_2|^2) \left[\begin{array}{cc} y_3-x_1y_2 & x_1 - y_1  \\ \bar x_1 y_3-y_2 & 1 - \bar x_1y_1 \end{array}\right] J \left[\begin{array}{cc} y_3-x_1y_2 & x_1 - y_1  \\ \bar x_1 y_3-y_2 & 1 - \bar x_1y_1 \end{array}\right]^*.
\]
The entries in the second column of this identity give us
\begin{eqnarray*}
 y'_1-\bar y'_2y'_3 &=& |\zeta|^2(1-|x_2|^2)\{(x_1-y_1)(1-x_1\bar y_1) - (y_3-x_1y_2)(x_1\bar y_3 - \bar y_2)\},\\
   &=&|\zeta|^2(1-|x_2|^2)\{ -y_1+\bar y_2y_3 +x_1(1+|y_1|^2-|y_2|^2-|y_3|^2)-x_1^2(\bar y_1-y_2\bar y_3)\},\\
1-|y'_2|^2 &=& |\zeta|^2(1-|x_2|^2)\{|1-x_1\bar y_1|^2 - |x_1\bar y_3 - \bar y_2|^2\}.
\end{eqnarray*}
On substituting these formulae and their symmetric analogues into the criterion (\ref{criter}) we obtain the statement in the lemma.
\end{proof}
Here is a less concrete but more assimilable version of this result.
\begin{Corollary} \label{KCd}
If $x,y \in\E$ and at least one of $x,y$ is a triangular point then
\[
\mathcal{C}_\E(x,y)=\mathcal{K}_\E(x,y)=\delta_\E(x,y). 
\]
\end{Corollary}
The result is immediate from Corollary \ref{KandC}, the invariance of $\mathcal{C}_\E,\mathcal{K}_\E$ and $\delta_\E$ under automorphisms and Remark \ref{orbitof0}.

 \section{The distinguished boundary of the tetrablock} \label{shilov}
Let $\Omega$ be a domain in $\C^n$ with closure $\bar\Omega$ and let $A(\Omega)$ be the algebra of continuous scalar functions on $\bar\Omega$ that are holomorphic on $\Omega$.  A {\em boundary} for $\Omega$ is a subset $C$ of $\bar\Omega$ such that every function in $ A(\Omega)$ attains its maximum modulus on $C$.  It follows from the  
theory of uniform algebras  \cite[Corollary 2.2.10]{browder} that (at least when $\bar\Omega$ is polynomially convex, as in the case of $\E$) there is a smallest closed boundary of $\Omega$, contained in all the closed boundaries of $\Omega$ and called the {\em distinguished boundary} of $\Omega$ (or the {\em Shilov boundary} of $A(\Omega)$).  In this section we shall determine the distinguished boundary of $\E$; we denote it by $\dE$.

Clearly, if there is a function $g\in A(\E)$  and a point $p \in \bar\E$ such that $g(p) = 1$ and $ | g(x) | < 1$ for all $x \in \bar\E \setminus\{p\}$, then $p$ must belong to $\dE$. We call such a point $p$ a {\em peak point} of $\bar\E$ and the function $g$ a {\em peaking function} for $p$. 
  
 An {\em analytic disc} in $\bar\E$  is a non-constant analytic function
  $f:\D \to \bar\E$.  It follows easily from the maximum modulus principle that no element of the image $f(\D)$ can be a peak point. 
\begin{Theorem} \label{characdE}
For $x \in \C^3$ the following are equivalent.
\begin{enumerate}
\item[(1)] $x_1 = \bar x_2 x_3, |x_3| = 1 $ and $ |x_2| \leq 1$;
\item[(2)] either $x_1 x_2 \neq x_3$ and $\Psi(\cdot, x)$ is an automorphism of $\D$ or $x_1 x_2 = x_3$ and $|x_1| = |x_2| = |x_3| =1$;
 \item[(3)]  $x $ is a peak point of $\bar\E$;
\item[(4)] there exists a $2 \times 2$ unitary matrix $U$ such that $x=\pi(U)$;
\item[(5)] there exists a symmetric $2 \times 2$ unitary matrix $U$ such that $x=\pi(U)$;
\item[(6)]  $x \in \dE$, the distinguished boundary of $\E$
\item[(7)] $x\in\bar\E$ and $|x_3|=1$.
\end{enumerate}
\end{Theorem}
\begin{proof}
 We first prove the equivalence of conditions (1) to (5); the proof is most easily presented as two completely separate cases. We first consider the simpler case $x_1 x_2 = x_3$.
 We show that each of the conditions is equivalent to the applicable part of~(2): $x_1 x_2 = x_3$ and $|x_1| = |x_2| = |x_3| =1$.

\noindent
(1)$\Leftrightarrow$(2) If (1) holds for a triangular point $(x_1, x_2, x_3)$ then $|x_1x_2|=|x_3|=1$, and since $|x_1|  \le 1,|x_2| \leq 1$ we must have 
 $|x_1| = |x_2| = |x_3| =1$, and hence (2) holds.  The converse is elementary.

\noindent (2)$\Rightarrow$(3) Let $x$ satisfy (2). Define $g:\bar\E \to \C$ by 
\[
  g(y_1, y_2, y_3) = (\bar x_1 y_1 + \bar x_2 y_2 + \bar x_3 y_3 + 1)/ 4.
\]
 For  $y$ in $\bar\E$ we have $|y_i | \leq 1, i=1,2,3$, and so $| g(y) | \leq 1$. Further, if $| g(y) | =1$, then each $\bar x_i y_i$ must be 1 and so $y=x$. This shows that $g$ is a peaking function for $x$ relative to $\bar\E$.  Hence $x$ is a peak point of  $\bar\E$.

\noindent (3)$\Rightarrow$(2) Consider a triangular point $x$ that is a peak point of $\bar\E$.  Suppose that $|x_3| < 1$.  Note that $|x_1x_2| < 1$ and so either $| x_1 | <  1$ or $| x_2 | <  1$. We assume that $| x_1 | <  1$. Consider the function given by $g(z) = (z, x_2, z x_2)$. It follows from condition (3) of Theorem~\ref{closE} that $g(\D) \subset \bar\E$. Thus $g$ is an analytic disc in $\bar\E$ which contains the point $x$.  This contradicts the hypothesis that  $x$ is a peak point, and so we have $|x_3|=1.$  Since $|x_1x_2|=1$ it is also true that $|x_1|=|x_2|= 1$.

\noindent (2)$\Rightarrow$(5)$\Rightarrow$(4)$\Rightarrow$(2). If $x$ satisfies (2) then it is clear that the diagonal matrix  $U=\diag(x_1, x_2)$ satisfies condition (5).  Trivially (5)$\Rightarrow$(4). Any unitary which is triangular is diagonal and hence (4) implies (2).

Thus conditions (1) to (5) are equivalent in the triangular case.

 Now consider the non-triangular case, $x_1 x_2 \neq x_3$.  Note that if $x\in\bar\E$ then $|x_1| <1$ and $|x_2| <1$, for otherwise conditions ($3$) and ($3'$) of Theorem~\ref{closE} show that  $x\in\bar\E$ is triangular. 
 
\noindent (1)$\Leftrightarrow$(2)
 If $x$ satisfies (1) then $\bar x_1x_3=x_2\bar x_3 x_3=x_2$ and so
 \[ 
 \Psi(z,x) = \frac{ x_3 z - x_1 }{ x_2 z - 1} =  \frac{ x_3 z - x_3 \bar x_2 }{ x_2 z - 1} 
 =  x_3 \frac{ z -  \bar x_2 }{ x_2 z - 1} 
 \]
 and $\Psi(.,x)$  is an automorphism of $\D$.
  Conversely, if $ \Psi(.,x)$  is an automorphism of $\D$ then $x\in\bar\E$ and the image $\Psi(\D,x)$ has centre~0 and radius~1. As we noted in equation~(\ref{imPsi}) the centre is $(x_1 - \bar x_2 x_3)/(1 - |x_2|^2)$ and  the radius is 
 $|x_1 x_2 - x_3| /( 1 - |x_2|^2)$.
 Thus $x_1=\bar x_2x_3, x_2 = \bar x_1 x_3$ and 
  $|x_2 \bar x_2 x_3 - x_3| /( 1 - |x_2|^2) = |x_3| = 1$.  Hence (2) implies (1).
  
 \noindent (4)$\Rightarrow$(1)$\Rightarrow$(5)$\Rightarrow$(4)
Suppose (4): there exists a unitary matrix 
 \begin{equation*}
 U = \left[ \begin{array}{cc} x_1 &  b \\ c & x_2
\end{array} \right],
\end{equation*}
such that $\det U = x_1 x_2 - b c = x_3$. It is immediate  that $|x_3| = |\det U| =1$ and
$|x_1| \leq ||U|| = 1, |x_2| \leq 1$.  Since the columns of $U$ are orthonormal $x_1\bar b + c\bar x_2 =0$ and so
\[
0=b(x_1\bar b + c\bar x_2) = x_1|b|^2 + bc\bar x_2 = x_1(1-|x_2|^2) +(x_1x_2 - x_3) \bar x_2 = x_1 - \bar x_2 x_3.
\]
Thus (4)$\Rightarrow$(1).  Suppose (1) holds.  Let $\zeta\in\T$ be a square root of $-x_3$. Then $\bar x_1\zeta + \bar\zeta x_2 =0$ and so
  \begin{equation*}
 U = \left[ \begin{array}{cc} x_1 &  \zeta \sqrt{1 - |x_2|^2} \\ \zeta\sqrt{1 - |x_2|^2}  & x_2
\end{array} \right]
\end{equation*}
 is a symmetric unitary matrix satisfying the  conditions of (5).  Trivially (5)$\Rightarrow$(4).
 
\noindent (2)$\Rightarrow$(3)
 Suppose $x$ satisfies (2) (and is non-triangular).
 We will exhibit a peaking function for $x$. 
  Let $\up$ be the inverse of the automorphism $\Psi(.,x)$  of
   $\D$.  Since
\[
\Psi(.,\up\cdot x)=\Psi(.,\tau(\up)\dia x) = \up \circ \Psi(.,x) = \mathrm{id}_\D = \Psi(.,0,0,-1),
\]
it follows that $\up\cdot x = (0,0,-1)$.

There is a natural right action of $\Aut\D$ on $A(\E)$:  for $\chi\in \Aut\D, g\in A(\E)$,
\[
g\cdot \chi(x) = g(\chi\cdot x) = g(\tau(\chi)\dia x).
\] 
If $g$ is a peaking function for a point $y\in\bar\E$ then $g\cdot\chi^{-1}$ is a peaking function for $\chi\cdot y$.  Thus it suffices to find a peaking function for the point $(0,0,-1)$. Consider the function $g(y) =  ((y_3 - y_1 y_2) - 1)/2$ on $\bar\E$. It follows from condition (3) of Theorem~\ref{closE}, that $| y_3 - y_1 y_2| \leq 1$ and hence
     $| g(y) | \leq 1 $ on $\bar\E$. Certainly $|g(0,0,-1)|=1$, and if
 $ | g(y) | = 1 $, then we must have
   $y_3 - y_1 y_2 = -1$ and, again by condition (3) of Theorem~\ref{closE}, 
    $ | y_1 |^2 =  | y_2 |^2 = 0$. Thus 
    $y=(0,0,-1)$, and hence  $g$ peaks at the point $(0,0,-1)$.  Consequently $x$ is a peak point and (2)$\Rightarrow$(3).
 
(3)$\Rightarrow$(2)  Suppose the non-triangular point $x$ is a peak point of $\bar\E$ but $\Psi(.,x)$ is not an automorphism of $\D$: we shall show that $x$ lies on an analytic disc in $\bar\E$ and obtain a contradiction.   The conclusion is trivial if $x\in\E$, and so we can assume that $x\in\partial\E$, the topological boundary of $\E$.  By condition (2) of Theorems~\ref{closE} and \ref{characE},  $\| \Psi \|_{H^\infty} = 1$ and so the closed disc $\Psi(\Delta,x)$ is a proper subset of $\Delta$ that touches $\T$ at a unique point, $\zeta$ say, so that $\Psi(\eta,x) = \zeta$ for some $\eta\in\T$.  Let us make use of the Cayley transform
\[
C_\eta(z) = \frac{\eta+z}{\eta-z},
\]
which maps $\Delta$ to the closed right half plane $\C_+$ and maps $\eta$ to $\infty$.
The (non-constant) M\"obius transformation $C_\zeta \circ\Psi(.,x)\circ C_\eta^{-1}$ maps $\C_+$ to a proper subset of itself and fixes $\infty$; it follows that 
\[
C_\zeta \circ\Psi(.,x)\circ C_\eta^{-1}(z) = az+b
\]
 for some $a>0$ and $b$ such that $\mathrm{Re}~ b >0$.  Let $F(z,w) = az + b + w\mathrm{Re}~ b.$  For each $w\in\D,\  F(.,w)$ is non-constant and maps $\C_+$ to a proper subset of itself.  Thus $C_\zeta^{-1}\circ F(.,w)\circ C_\eta$ is a non-constant M\"obius transformation that maps $\Delta$ to itself, hence can be written $\Psi(.,f(w))$ for some $f(w)\in\bar\E$.  The map $f$ is rational and so is an analytic disc in $\bar\E$, and $f(0)=x$.  This is a contradiction and so 
(3)$\Rightarrow$(2).

We have proved the equivalence of conditions (1) to (5) in both the triangular and non-triangular cases.  Next we show that (3)$\Leftrightarrow$(6).  According to \cite[Theorem 2.3.5]{browder}, (because $\bar\E$ is a metric space) the set $P$ of peak points of $\bar\E$ is a boundary for $\E$; it is clearly contained in every boundary of $\E$, so that (3)$\Rightarrow$(6).  By the equivalence of (1) and (3), $P$ is closed in $\bar\E$.  Hence $P$ is the smallest closed boundary of $\E$, that is $P=\dE$ and so (6)$\Rightarrow$(3). \\ 
(1)$\Leftrightarrow$(7)  If (1) holds then, by condition (3) of Theorem \ref{closE}, $x\in\bar\E$ and hence (7) holds.  If (7) holds then, by condition (6) of Theorem \ref{closE}, $x_1=\bar x_2x_3$, while by condition (3) of the same theorem $|x_2|\leq 1$.  Thus (7)$\Rightarrow$(1).
\end{proof}
\begin{Corollary}
$\dE$ is homeomorphic to $\Delta \times \T$.
\end{Corollary}
For the map $\Delta \times \T \to \dE:(x_2,x_3) \mapsto (\bar x_2 x_3, x_2, x_3)$ is a homeomorphism.
\begin{Corollary}
$\dE$ is the closure of $\Aut\D$ in $\bar\E$.
\end{Corollary}
\begin{proof}
By the definition (\ref{deftau}), the monomorphism $\tau$ identifies an automorphism of $\D$ with a point $(\omega \al, \bar\al,\omega)$ with $\omega\in\T, \al\in\D$.  Thus $\tau$ identifies $\Aut\D$ with the set $\{x:x_1=\bar x_2x_3, |x_2|<1,|x_3|=1\}$, which is clearly a dense subset of $\dE$.
\end{proof}

\section{The analytic retraction problem} \label{retract}
 We have seen in Corollary \ref{KCd} that $C_\E, K_\E$ and $\delta_\E$ all agree at any pair of points of which one is triangular.  Is it true that $C_\E=K_\E=\delta_\E$?  Note that the analogous equality holds for any convex domain, by a theorem of Lempert \cite{JP}, and so in particular for the convex domain $R_I(2,2)$, the unit ball of the space of $2\times 2$ complex matrices.  We have also seen that $\E$ is closely related to $R_I(2,2)$, by the analytic surjection $\pi:R_I\to \E$.
We ask: is $\E$ an analytic retract of $R_I(2,2)$?  In other words, do there exist analytic maps $h:\E\to R_I$ and $f: R_I\to \E$ such that $f\circ h= \mathrm{id}_\E$?  If the answer is yes then it follows that $C_\E=K_\E=\delta_\E$, since the following observation is a consequence of the fact that analytic maps are contractive for $C$ and $\delta$.
\begin{Lemma}
Let  $E,B$ be domains and let $E$ be an analytic retract of $B$. \\
{\rm (1)} $C_E= C_B|E$ and $\delta_E=\delta_B|E$.\\
 {\rm (2)} If  $C_B=\delta_B$ then $C_E=\delta_E$.
\end{Lemma}

We could therefore resolve the question (does $C_\E=\delta_\E$?) if we could find an analytic $h:\E\to R_I$ such that $\pi\circ h = \mathrm{id}_\E$.  In fact there is no such $h$, and we conjecture that $\E$ is not an analytic retract of $R_I(2,2)$.
\begin{Theorem} \label{noretract}
The map $\pi:R_I(2,2) \to \E$ has no analytic right inverse.
\end{Theorem}
\begin{proof}
Suppose $h=[h_{ij}]:\E \to R_I$ satisfies $\pi\circ h=\mathrm{id}_\E$.
Then $h_{11}(x)=x_1, \ h_{22}(x) = x_2$ and $h_{12}h_{21}(x) = x_1x_2-x_3$.
Let us write $P(x)= x_1x_2-x_3.$   Since $P$ is an irreducible polynomial and $h_{12} h_{21} = P$, it follows that $P$ divides one of $h_{12}, h_{21}$ -- say $h_{21} = Pg$ and hence $h_{12}= 1/g$ where $g, 1/g$ are analytic scalar functions on $\E$ and $Pg, 1/g \in H^\infty$.   By equation (\ref{formdet}), for any $x \in \E$,
\begin{eqnarray} \label{ing}
0 < \det(1-h(x)^*h(x)) &=& 1-|x_1|^2 - |x_2|^2 +|x_3|^2 -|x_1x_2-x_3|^2|g(x)|^2 -\frac{1}{|g(x)|^2} \nn \\
	&\leq& 1-|x_1|^2 - |x_2|^2 +|x_3|^2 -2|x_1x_2-x_3|.
\end{eqnarray}
Consider any non-triangular point $y \in \partial\E$.  By conditions (5) of Theorems \ref{characE} and \ref{closE},
\[
1-|y_1|^2-|y_2|^2+|y_3|^2 - 2|y_1y_2 -y_3| =0
\]
and therefore 
\[
1-|x_1|^2 - |x_2|^2 +|x_3|^2 -2|x_1x_2-x_3| \to 0 \mbox {  as  }x\to y, \ x\in\E.
\]
It follows from inequalities (\ref{ing}) that
\[
|x_1x_2-x_3|^2 |g(x)|^2 - \frac{1}{|g(x)|^2} -2|x_1x_2-x_3| \to 0 \mbox{ as }x \to y.
\]
By a refinement of the inequality of the means,
\[
|x_1x_2-x_3|^2 |g(x)|^2 - \frac{1}{|g(x)|^2}  \to 0 \mbox{ as }x \to y.
\]
Since $1/P(x)$ tends to the finite limit $1/P(y)$ as $x \to y$, we have
\begin{equation*} 
|g(x)|^2 - \frac{1}{|P(x)g(x)|^2}  \to 0 \mbox{ as }x \to y,
\end{equation*}
and since $Pg$ is bounded on $\E$,
\begin{equation} \label{Pgsq}
|P(x)^2g(x)^4|  \to 1 \mbox{ as }x \to y.
\end{equation}
Fix non-zero $\beta_1,\beta_2$ such that $|\beta_1|+|\beta_2| < 1$ and let $\beta=(\beta_1,\beta_2$).  We shall restrict the relation (\ref{Pgsq}) to the disc $\ph_\beta(\D) \subset \E$ and obtain a contradiction.

We claim that $\ph_\beta(\D)$ contains a unique triangular point of $\E$.  Indeed, $\ph_\beta(\la)$ is triangular if and only if 
\[
(\beta_1+\bar\beta_2\la)(\beta_2+\bar\beta_1\la)=\la,
\]
or 
\begin{equation} \label{quadeq}
\la^2 - \frac{1-|\beta_1|^2-|\beta_2|^2}{\bar\beta_1\bar\beta_2}\la +\frac{\beta_1\beta_2}{\bar\beta_1\bar\beta_2} =0.
\end{equation}
If the roots of this equation are $\la_1,\la_2$ then $|\la_1\la_2|=1$ and
\[
|\la_1 +\la_2| -2 = \frac{1-|\beta_1|^2-|\beta_2|^2}{|\beta_1\beta_2|} -2 = \frac{1-(|\beta_1|+|\beta_2|)^2}{|\beta_1\beta_2|} >0,
\]
so that exactly one of $\la_1,\la_2$ belongs to $\D$ -- say $|\la_1|<1, |\la_2|>1$.  Note also that $\ph_\beta(\T) \subset \partial\E$ contains no triangular points.

Write down the explicit inner-outer factorisation of $P\circ \ph_\beta$:
\[
P\circ\ph_\beta(\la)= \bar\beta_1\bar\beta_2(\la-\la_1)(\la-\la_2) = \upsilon_\beta(\la) q_\beta(\la)
\]
where
\[
\up_\beta(\la)=\frac{\la-\la_1}{\bar\la_1\la-1},\quad q_\beta(\la)=\bar\beta_1\bar\beta_2(\bar\la_1\la-1)(\la-\la_2).
\]
Observe that $q_\beta$ is bounded away from zero on $\D$.  Let
\[
\psi_\beta= q_\beta (g\circ \ph_\beta)^2.
\]
Since $g\circ\ph_\beta$ is analytic on $\D$ and $\up_\beta q_\beta g\circ\ph_\beta \in H^\infty$, it follows that $g\circ\ph_\beta \in H^\infty$.  Hence both $\psi_\beta$ and $1/\psi_\beta \in H^\infty$.  Moreover, by relation (\ref{Pgsq}),
\[
|\psi_\beta(\la)|^2  \to 1 \mbox{  as  } \la\to \omega \in \T.
\]
Thus the radial limits of $\psi_\beta$ have modulus $1$ everywhere on $\T$, so that $\psi_\beta$ is inner.  Since $1/\psi_\beta \in H^\infty$, $\psi_\beta$ is constant, and hence
\[
 1=|\psi_\beta(0)| = |q_\beta(0) g\circ\ph_\beta(0)^2| = |\bar\beta_1\bar\beta_2 \la_2 g(\beta_1,\beta_2,0)^2|
\]
and therefore
\[
 |\beta_1\beta_2 g(\beta_1,\beta_2,0)^2| = \frac{1}{|\la_2|} = |\la_1|.
\]
Thus $|\la_1|$ is the modulus of an analytic function  of $\beta$ on the domain $\{(\beta_1,\beta_2): \beta_1 \neq 0, \beta_2 \neq 0, |\beta_1|+|\beta_2|<1\}$,
and hence $\log |\la_1|$ is a pluriharmonic function on this domain.

  Let $u(z)$ be the unique root in $\D$ of the quadratic equation (\ref{quadeq}) with 
$\beta_1=\beta_2= z$.  On the planar domain $0<|z|<\tfrac 12$, $u$ does not vanish and $\log|u(.)|$ is a harmonic function.  We have
\[
\bar z^2 u(z)^2 - (1-2z\bar z)u(z) + z^2 =0.
\]
Implicit differentiation of this relation (together with the implicit function theorem) yields
\[
\frac{\partial u}{\partial \bar z} = \frac{ -2u(\bar zu+z)}{2\bar z^2u-1+2z\bar z} = 
 u\frac{\partial u}{\partial  z}
\]
(one can check that the denominator of the middle term never vanishes when $|z| < \tfrac 12$),
\[
\frac{\partial^2 u}{\partial \bar z\partial z} = \frac{2(u-4z\bar z u - 4z^2)}{(2\bar z^2u -1 + 2z\bar z)^3}.
\]
Thus
\begin{eqnarray*}
 \frac{\partial^2 \log |u(z)|}{\partial \bar z\partial z} &=& \mathrm{Re} \left\{
\frac{1}{u} \frac{\partial^2 u}{\partial \bar z\partial z} - \frac{1}{u^2} \frac{\partial u}{\partial \bar z}\frac{\partial u}{\partial  z} \right\} \\
	&=& 8\mathrm{Re} \left\{ \frac{u - 2z\bar z u - \bar z^2u- z^2}{u(2\bar z^2u -1+2z\bar z)^3} \right\}.
\end{eqnarray*}
The right hand side is not identically zero: for example, it is non-zero at the point $z=\tfrac 13, u= \frac{7-3\sqrt{5}}{2}$.  This is a contradiction, and so the postulated analytic $h:\E\to R_I$ does not exist.
\end{proof}
\begin{Remark} \rm
{\em A fortiori} $\pi: R_{II}(2)\to\E$ has no right inverse either.
\end{Remark}

\section{Relation to the $\mu$-synthesis problem} \label{mu}
In the theory of robust control the {\em structured singular value} of an $m \times n$ matrix $A$, denoted by $\mu(A)$, is a cost function that generalizes the usual operator norm of $A$ and encodes structural information about the perturbations of $A$ that are being studied.  In this context a ``structure'' is identified with a linear subspace of $\C^{n \times m}$.  Let 
$E$ be such a subspace, and write
\[
\mu_E(A) = \left( \inf \{ ||X|| : X \in E, 1-AX  \mbox{ is singular} \} \right)^{-1},
\]
where we adopt the natural interpretation that $\mu_E(A)=0$ in the event that $1-AX$ is non-singular for all $X \in E$.  If $E =\C^{n \times m}$ then $\mu_E= ||.||$, while if $m=n$ and $E$ is the space of scalar multiples of the identity matrix then $\mu_E$ is the spectral radius.  For a given $E \subset \C^{n \times m}$ the $\mu$-synthesis problem is to construct, if possible, an analytic $m \times n$-matrix-valued function on $\D$ or the right half plane subject to a finite number of interpolation conditions such that
\[
\mu_E(F(\la)) \leq 1 \mbox{ for all $\la$ in the domain of $F$}.
\]
In the case that $\mu_E= ||.||$ the $\mu$-synthesis problem is the classical Nevanlinna-Pick problem, for which there is a detailed theory (e.g. \cite{BGR}).  More generally,
engineers have had some success in computing numerical solutions of $\mu$-synthesis problems \cite{matlab}, but there is a dearth of convergence results and existence theorems.  At present there is not even a sufficient theory to enable the numerical methods to be tested satisfactorily.  There is a clear need for a better understanding of the solvability or otherwise of $\mu$-synthesis problems.  Bercovici, Foia\c{s} and Tannenbaum \cite{BFT0,BFT1,BFT2} obtained some solvability criteria with the aid of variants of the commutant lifting theorem; however, the criteria they obtained are not easy to check.  A solution for the special case of the spectral Nevanlinna-Pick problem (that is, with $\mu_E$ being the spectral radius) for $2\times 2$ matrix functions and $2$ interpolation points follows from the theory  \cite{AY1, AY2}  of the symmetrised bidisc.
For more than two interpolation points even this very special case of $\mu$-synthesis is not yet well understood.

In the engineering literature (for example \cite{DP}) the space $E$ of matrices is usually taken to be given by a block diagonal structure.  If we confine ourselves to $2\times 2$ matrices the next natural case for study is that of the space of diagonal matrices:
\[
E= \mathrm{Diag} \stackrel{\rm def}{=} \{ \diag(z,w): z,w \in \C \}.
\]
This paper arises out of a study of the $\mu$-synthesis problem for $2\times 2$ matrices
where $\mu=\mu_\mathrm{Diag}$.  There is a simple connection between $\E$ and the set of matrices for which $\mu < 1$.
\begin{Theorem} \label{mult1}
An element $x$ of $\C^3$ belongs to $\E$ if and only if there exists 
$A \in \C^{2\times 2}$ such that $\mu_{\mathrm{Diag}} (A)< 1$ and $x = \pi(A).$  Similarly, $x \in \bar\E$ if and only if there exists 
$A \in \C^{2\times 2}$ such that $\mu_{\mathrm{Diag}} (A)\leq 1$ and $x = \pi(A).$ 
\end{Theorem}
Henceforth we shall write $\mu$ for $\mu_\mathrm{Diag}$.
\begin{proof}
For $r >0$ and $A =[a_{ij}]\in \C^{2\times 2}$ observe that $\mu(A) \leq 1/r$ if and only if 
$||X|| \ge r$ whenever $X \in \mathrm{Diag}$ and $\det(1-AX)=0$.  If $X=\diag(z,w)$
then
\begin{eqnarray} \label{det}
\det(1-AX) &=& (1-a_{11}z)(1-a_{22}w) - a_{12}a_{21}zw \nonumber \\
	&=& 1 - a_{11}z - a_{22}w - (\det A)zw.
\end{eqnarray}
Thus $\mu(A) \leq 1/r$ if and only if the zero variety of the polynomial (\ref{det})
in $z,w$ does not meet the open bidisc $r\D \times r\D$.

Suppose that $\mu(A) < 1$ and $x=(a_{11},a_{22}, \det A)$.  For some $r > 1$ we have $\mu(A) \le 1/r$, and so the zero variety of the polynomial (\ref{det}) is disjoint from
$(r\D)^2$, hence {\em a fortiori} from $\Delta^2$.  Thus $x \in \E$.

Conversely, if $x \in \E$, then the zero variety of (\ref{det}) is disjoint from $(r\D)^2$ for some $r > 1$, and hence the matrix
\[
A=\left[ \begin{array}{cc} x_1 & x_1 x_2 -x_3 \\ 1 & x_2 \end{array} \right]
\]
satisfies $\mu(A) < 1$ and $x=\pi(A)$.

The proof of the second statement is similar.
\end{proof}
We have found the bounded $3$-dimensional domain $\E$ more amenable to study than the unbounded $4$-dimensional domain
\[
\Sigma \stackrel{\rm def}{=} \{ A \in \C^{2\times 2}: \mu(A) < 1\}.
\]
Every analytic function $F:\D \to \Sigma$ induces an analytic function $\pi \circ F:
\D \to \E$, where $\pi $ is the map defined in equation (\ref{defpi})
having the property that $A \in \Sigma$ if and only if $\pi(A) \in \E$.
Conversely, every analytic $\varphi: \D \to \E$ lifts to a map $F:\D \to \Sigma$ such that
$\pi \circ F = \varphi$, for example,
\[
F = \left[ \begin{array}{cc} \varphi_1 & \varphi_1\varphi_2 - \varphi_3\\
	1 & \varphi_2 \end{array} \right].
\]
More is true: the interpolation problems for $\Sigma$ and $\E$ are equivalent in the following sense.
\begin{Theorem} \label {Eandsigma}
Let $\la_1, \dots, \la_n$ be distinct points in $\D$ and let $A_k=[a_{ij}^k] \in \Sigma,
\, 1 \leq k \leq n.$  The following conditions are equivalent.
\begin{enumerate}
\item[\rm (1)] There exists an analytic function $F:\D \to \Sigma$ such that
$F(\la_k)= A_k, \, 1\leq k \leq n$;
\item[\rm(2)] there exists an analytic function $\varphi: \D \to \E$ such that 
$\varphi(\la_k) = \pi(A_k)$ and, if $A_k$ is a diagonal matrix, then 
\[
\varphi_3^\prime(\la_k) = a_{22}^k \varphi_1^\prime(\la_k) + a_{11}^k \varphi_2^\prime(\la_k), \, 1\leq k \leq n.
\]
\end{enumerate}
\end{Theorem}
\begin{proof}
(1) $\Rightarrow$ (2) is easy, for we may take $\varphi = \pi \circ F$.  
Then $\ph_3 = \ph_1 \ph_2 - F_{12}F_{21}$ and $\ph_1(\la_k)=a_{11}^k, \, \ph_2(\la_k)=a_{22}^k.$
If $a_{12}^k = a_{21}^k = 0$ then $F_{12}(\la_k)=F_{21}(\la_k) = 0$ and
\begin{eqnarray*}
\ph_3^\prime (\la_k) &=& \ph_1^\prime \ph_2(\la_k) + \ph_1 \ph_2^\prime(\la_k) - F_{12}^\prime(\la_k)F_{21}(\la_k) - F_{12}(\la_k)F_{21}^\prime(\la_k)\\
	&=& a_{22}^k \ph_1^\prime(\la_k) + a_{11}^k \ph_2^\prime(\la_k).
\end{eqnarray*}

(2)$\Rightarrow$(1)  Let $\ph$ be as in (2), so that
\[
(\ph_1\ph_2 - \ph_3)(\la_k) = a_{11}^k a_{22}^k - \det A_k = a_{12}^k a_{21}^k
\]
and, if $a_{12}^k = a_{21}^k =0$, then $(\ph_1\ph_2 - \ph_3)^\prime(\la_k) =0.$
Choose an analytic function $g$ in $\D$ such that
\begin{enumerate}
\item[\rm(i)] $g(\la_k) = a_{21}^k, \, 1 \leq k \leq n$;
\item[\rm(ii)] $g$ has simple zeros at those $\la_k$ such that $a_{21}^k = 0$ and no other zeros;
\item[\rm(iii)] if $a_{21}^k = 0$ and $a_{12}^k \neq 0$ then 
$g^\prime (\la_k) = (\ph_1 \ph_2 - \ph_3)^\prime(\la_k)/a_{12}^k.$
\end{enumerate}
Let $f= (\ph_1 \ph_2 - \ph_3)/g$ and let
\[
F= \left[ \begin{array}{cc} \ph_1 & f \\ g & \ph_2 \end{array} \right] .
\]
$F$ is analytic in $\D$ and $\pi \circ F = \ph$, so that $F(\D) \subset \Sigma.$
Note that, if $A_k$ is diagonal, then $\ph_1 \ph_2 - \ph_3$ has a multiple zero at
$\la_k$ and $g$ has a simple zero, so that $f(\la_k) = 0 = a_{12}^k$.  If $a_{21}^k =0$ and $a_{12} \neq 0$ then L'Hopital's rule gives $f(\la_k) = a_{12}^k.$  Hence
$ F(\la_k) = A_k, \, 1 \leq k \leq n$.
\end{proof}
{\begin{Remark} \label{closures} \rm
(i) The theorem remains true if we replace $\Sigma$ and $\E$ by their closures.\\
(ii) Generically the target matrices are non-diagonal, in which case the interpolation problem for $\ph:\D \to \E$ does not involve conditions on $\ph^\prime$.\\
(iii) The problem of finding $F$ satisfying (1) in Theorem \ref{Eandsigma} is called the {\em structured Nevanlinna-Pick problem} in \cite{BFT0}.
\end{Remark}
On putting together Theorems \ref{schwarzL} and \ref{Eandsigma} we obtain a Schwarz lemma for $\bar\Sigma$.  
\begin{Theorem} \label{schwarzLSig}
Let $\la_0 \in \D \setminus \{0\}, \zeta \in \C$ and $ A_1, A_2 \in \C^{2\times 2}$, where  
\[
A_1 = \left[\begin{array}{cc} 0 & \zeta\\ 0&0 \end{array} \right]
\mbox{ or } \left[\begin{array}{cc} 0 & 0\\ \zeta & 0 \end{array} \right]
\mbox{ and  } \pi(A_2)=(a,b,p) \in \E.  
\]
There exists an analytic
$2\times 2$ matrix function $F$ such that 
$F(0) = A_1, \, F(\la_0) = A_2$ and $\mu(F(\la))\leq 1$ for all $\la \in \D$ if and only if
\begin{equation} \label{2cases}
\left\{\begin{array}{cl} \displaystyle
\max \left\{ \frac{|a-\bar b p|+|ab-p|}{1-|b|^2}, 
\frac{|b - \bar a p|+|ab-p|}{1-|a|^2} \right\} \leq |\la_0| & 
\mbox{ if } \zeta \neq 0\\ 
\displaystyle \left(\frac{a}{\la_0},\frac{b}{\la_0},\frac{p}{\la_0^2}\right) \in \bar\E &  \mbox{ if } \zeta = 0.
\end{array} \right . 
\end{equation}
\end{Theorem}
\begin{proof}
If $\zeta \neq 0$ then $A_1$ is not diagonal, and so by Theorem \ref{Eandsigma} and Remark \ref{closures}(i),  there is a function $F$ with the required properties if and only if there exists an
analytic function $\ph:\D \to \bar\E$ such that $\ph(0)= (0,0,0)$ and 
$\ph(\la_0) = (a,b,p)$.  By Theorem \ref{schwarzL} this is so if and only if the
first inequality in conditions (\ref{2cases}) holds.

If $\zeta =0$ then $A_1=0$ and the desired $F$ exists if and only if there is a function
$G$ in the $2\times 2$ Schur class such that $G(\la_0)=A_2/\la_0$, which is so if and only if $\|A_2/\la_0\| \leq 1$, and so, by condition (7) of Theorem \ref{closE}, if and only if 
$(a/\la_0,b/\la_0, p/\la_0^2) \in \bar\E$.
\end{proof}
{\begin{Remark} \label{muschw} \rm
(i)  This result is a solvability criterion for an extremely special type of 
2-point $\mu$-synthesis problem.  It falls far short of what control engineers would like to know, but it does reveal some of the analytic subtleties  of $\mu$-synthesis and may be a starting point for the solution of more general problems.  \\
(ii)  In control problems the interpolation conditions are typically ``tangential'', that is, of the forms $F(\la_j)x_j =y_j$ and $x_j^*F(\la_j)=y_j^*$ for suitable vectors $x_j, y_j$, rather than $F(\la_j)=A_j$ as studied here, but a solution of the general problem must of course include our type of constraint.\\
(iii) The condition that $(\frac{a}{\la_0},\frac{b}{\la_0},\frac{p}{\la_0^2}) \in \bar\E$
can  be written in terms of any of the criteria of Theorem \ref{closE}.  For example,
by condition (5), it is equivalent to
\[
|\la_0|^4 -(|a|^2+|b|^2+2|ab-p|) |\la_0|^2 +|p|^2 \ge 0
\]
and if $ab=p$ then $|a| + |b| \leq 2|\la_0|$.\\
(iv) Observe that a 2-point $\mu$-synthesis problem can be ill-conditioned.
 For example, if $(a,b,p) = (\tfrac 12, \tfrac 12, \tfrac 12)$, then there exists an analytic function $F_\zeta$ in $\D$ such that $\mu(F_\zeta(\la)) \leq 1$ for all $\la \in \D$ and
\[
F_\zeta(0) = A_1, \quad F_\zeta(\la_0) = A_2
\]
if and only if
\[
 |\la_0| \ge \left\{ \begin{array}{cl} \tfrac 23 & \mbox{ if } \zeta \neq 0 \\
	 & \\
	\tfrac 1{\sqrt{2} }& \mbox{ if } \zeta = 0. \end{array} \right.
\]
It follows that if  $\tfrac 23 < |\la_0| < \frac 1{\sqrt{2}}$  the $F_\zeta$ cannot be locally bounded as $\zeta \to 0$.  For such $\la_0$, if $\zeta$ is close to zero then the solutions of the interpolation problem are very sensitive to small changes in $\zeta$.  Any numerical method for the computation of solutions is likely to be unreliable for such data.\\
(v)  The proof shows how to construct a solution of a 2-point problem of the type in Theorem
\ref{schwarzLSig}, at least in the case that the inequality in conditions (\ref{2cases}) holds strictly.  For then, if $\zeta=0$, we may define $\ph =\pi\circ F$ where
$F(\la)= \la A_2/\la_0$, while if $\zeta \neq 0$ then we may take $\ph=\pi\circ F$
where $F$ is constructed according to the algorithm in Section \ref{schwarz}.
\end{Remark}
Similarly, by putting together Theorem \ref{Eandsigma} and Corollary \ref{schwpick} we obtain a partial Schwarz-Pick lemma for $\bar\Sigma$.
\begin{Theorem} \label{SPickSig}
Let $\la_1,\la_2$  be distinct points in $\D$, let $A, B$ be non-diagonal ${2\times 2}$ matrices such that 
$\mu(A)\leq 1, \mu(B)\leq 1$ and $A$ is triangular.  There exists an analytic $2\times 2$ matrix function $F$ on $\D$ such that $F(\la_1)=A, F(\la_2)=B$ and $\mu(F(\la))\leq 1$ for all $\la \in\D$ if and only if
\begin{eqnarray*}
  &\max &\left\{  \frac{(1-|x_1|^2)|y_3-y_1y_2|+|y_1-\bar y_2y_3-x_1(1+|y_1|^2-|y_2|^2-|y_3|^2)+x_1^2(\bar{y}_1-y_2\bar{y}_3)|}{|1-\bar x_1y_1|^2-|y_2-\bar x_1y_3|^2} , \right. \\
     &&\left.  \frac{(1-|x_2|^2)|y_3-y_1y_2|+|y_2-\bar y_1y_3-x_2(1-|y_1|^2+|y_2|^2-|y_3|^2)+x_2^2(\bar y_2-y_1\bar y_3)|}{|1-\bar{x}_2y_2|^2-|y_1-\bar{x}_2y_3|^2}   \right\}\\
 {} &&\leq \quad |d(\la_1,\la_2)|
\end{eqnarray*}
where $ \pi(A)=x, \pi(B)=y$.
\end{Theorem}
One can derive a somewhat more complicated criterion in the case of diagonal $A$.

  Bercovici, Foia\c{s} and Tannenbaum \cite{BFT0} use operator-theoretic methods to study a much more general $\mu$-synthesis problem than the special cases in Theorems \ref{schwarzLSig} and \ref{SPickSig}, but they obtain a less detailed result.  For the purpose of comparison we shall state their result, specialised to the situation we are studying here ($2\times 2$-matrix functions, $\mu=\mu_\mathrm{Diag}$).

Suppose we are given distinct points $\la_1, \dots, \la_n \in \D$ and $2\times 2$ matrices $A_1, \dots, A_n$.  For any analytic function $F:\D\to\C^{2\times 2}$ let
\[
 \mu^\infty (F) = \sup_{z\in\D} \mu(F(z)).
\]
We seek to minimise $\mu^\infty(F)$ over all analytic interpolating functions $F$; the formula is in terms of operators.  Let $k_\la$ be the Szeg\H{o} kernel:
\[
k_\la(z) = \frac{1}{1-\bar\la z}, \quad   z\in  \D,
\]
let $H^2$ denote the Hardy space on the disc and let
\[
\mathcal{M} = \mathrm{span}~\{k_{\la_1} \otimes \xi_1, \dots, k_{\la_n} \otimes \xi_n: \xi_1,\dots,\xi_n \in \C^2\},
\]
which is a $2n$-dimensional subspace of the Hilbert space $H^2\otimes\C^2$.  Corresponding to $2\times 2$ matrices $F_1,\dots,F_n$ we define a linear operator $A(F_1, \dots, F_n)$ on $\mathcal{M}$ by
\[
A(F_1,\dots,F_n)^* k_{\la_j} \otimes \xi = k_{\la_j} \otimes  F_j^*\xi.
\]
 Theorem 5 of  \cite{BFT0} states the following.
  {\em The infimum of $\mu^\infty(F)$ over all bounded rational analytic $2\times 2$ functions $F$ on $\D$ such that 
\[
F(\la_j) = A_j, \quad 1 \leq j \leq n,
\]
is equal to
\[
\inf \{ ||A(D_1A_1D_1^{-1}, \dots D_nA_nD_n^{-1})|| : D_1, \dots,D_n \in \mathrm{Diag}\cap GL_2(\C) \}.
\]}
This result gives the infimum over the infinite-dimensional set of $F$s in terms of an infimum over an $n$-dimensional set of $D$s; existing packages for the numerical solution of $\mu$-synthesis problems work by attempting to solve this $n$-dimensional (non-convex, unbounded) optimization problem. 
Note, however, that there is no assertion as to whether the infima are attained.

Both this paper and \cite{BFT0} seek to reduce $\mu$-synthesis problems to classical Nevanlinna-Pick problems, in one case via the introduction of $\E$ and the geometry of the Cartan domain $R_I(2,2)$, in the other directly by diagonal scaling.  We believe the two approaches complement each other, and that there is scope for further progress on $\mu$-synthesis problems through a study of $\E$ and possibly higher-dimensional analogues.

\vspace*{0.5cm}

{\em Added in proof:} Some of the questions raised in this paper are answered in the eprint ``The automorphism group of the tetrablock'', arXiv:0708.0689.  The automorphisms described in Section \ref{autos} do indeed comprise all automorphisms of $\E$, so that Conjecture \ref{allautos} is true.  It is also shown that $\E$ is not an analytic retract of $R_I(2,2)$ or $R_{II}(2)$.

\vspace*{0.5cm}

\nin {A. A. Abouhajar,
 School of Mathematics and Statistics,
Newcastle University,
Newcastle upon Tyne NE1 7RU,
  England}

\vspace*{0.5cm}

\nin {M. C. White,
   School of Mathematics and Statistics,
Newcastle University,
Newcastle upon Tyne NE1 7RU,
  England}
%{e-mail: michael.white@ncl.ac.uk }

\vspace*{0.5cm}

\nin {N. J. Young,
School of Mathematics,
Leeds University,
Leeds LS2 9JT,
England}

\end{document}